\documentclass[sn-mathphys,Numbered]{sn-jnl}

\usepackage{graphicx}%
\usepackage{multirow}%
\usepackage{amsmath,amssymb,amsfonts}%
\usepackage{amsthm}%
\usepackage{mathrsfs}%
\usepackage[title]{appendix}%
\usepackage{xcolor}%
\usepackage{textcomp}%
\usepackage{manyfoot}%
\usepackage{booktabs}%
\usepackage{algorithm}%
\usepackage{algorithmicx}%
\usepackage{algpseudocode}%
\usepackage{listings}%
\usepackage{tikz}
\usetikzlibrary{decorations.shapes}
\usepackage{subcaption}
\usepackage{enumerate}
\usepackage{hyperref}

\newcommand{\x}{\boldsymbol x}

\newcommand{\y}{\boldsymbol y}
\newcommand{\X}{\mathcal X}
\renewcommand{\emptyset}{\varnothing}
\renewcommand{\ge}{\geqslant}
\renewcommand{\geq}{\geqslant}
\renewcommand{\le}{\leqslant}
\renewcommand{\leq}{\leqslant}

\theoremstyle{thmstyleone}%
\newtheorem{theorem}{Theorem}
\newtheorem{proposition}[theorem]{Proposition}%
\newtheorem{corollary}[theorem]{Corollary} 

\theoremstyle{definition}%
\newtheorem{definition}[theorem]{Definition}%
\newtheorem{example}[theorem]{Example}%
\newtheorem{remark}[theorem]{Remark}%

\newtheorem{problem}[theorem]{Problem}
\usepackage{xskak}
\usepackage{tkz-euclide}
\usepackage{skak}

\usetikzlibrary{matrix}

\nocite{*}
\begin{document}
\title[Article Title]{Enforce and selective  operators of combinatorial games}

\author[1]{\fnm{Tomoaki} \sur{Abuku}}\email{buku3416@gmail.com}
\equalcont{These authors contributed equally to this work.}

\author[2]{\fnm{Shun-ichi} \sur{Kimura} }\email{skimura@hiroshima-u.ac.jp}
\equalcont{These authors contributed equally to this work.}

\author[3]{\fnm{Hironori} \sur{Kiya}}\email{kiya@omu.ac.jp}
\equalcont{These authors contributed equally to this work.}

\author[4]{\fnm{Urban} \sur{Larsson}}\email{larsson@iitb.ac.in}
\equalcont{These authors contributed equally to this work.}

\author*[5]{\fnm{Indrajit} \sur{Saha}}\email{indrajit@inf.kyushu-u.ac.jp}
\equalcont{These authors contributed equally to this work.}

\author[6]{\fnm{Koki} \sur{Suetsugu}}\email{suetsugu.koki@gmail.com}
\equalcont{These authors contributed equally to this work.}

\author[2]{\fnm{Takahiro} \sur{Yamashita}}\email{d236676@hiroshima-u.ac.jp}
\equalcont{These authors contributed equally to this work.}



\affil[1,6]{
\orgname{National Institute of Informatics}, 
\orgaddress{\street{Hitotsubashi}, \city{Chiyoda}, \postcode{101-8430}, \state{Tokyo}, \country{Japan}}}

\affil*[1]{\orgname{Gifu University}, \orgaddress{\street{1-1 Yanagido}}, \city{Gifu City}, \postcode{501-1193}, \state{Gifu}, \country{Japan}}

\affil[2]{
\orgdiv{Department of Mathematics},
\orgname{Hiroshima University}, 
\orgaddress{\street{Kagamiyama}, \city{Higashi-Hiroshima City}, \postcode{739-8526}, \state{Hiroshima}, \country{Japan}}}

\affil[3]{
\orgdiv{Department of Informatics}, 
\orgname{Osaka Metropolitan University}, 
\orgaddress{\street{Gakuencho, Nakaku}, 
\city{Sakai}, 
\postcode{599-8531}, 
\state{Osaka}, 
\country{Japan}}}

\affil[4]{
\orgdiv{Industrial Engineering and Operations Research},
\orgname{Indian Institute of Technology Bombay}, 
\orgaddress{\street{Powai}, \city{Mumbai}, \postcode{400076}, \state{Maharashtra}, \country{India}}}

\affil[5]{
\orgdiv{Department of Informatics},
\orgname{Kyushu University}, 
\orgaddress{\street{Motooka}, \city{Fukuoka}, \postcode{819-0395}, \state{Fukuoka}, \country{Japan}}}

\affil*[6]{\orgname{Waseda University}, \orgaddress{\street{513 Waseda-Tsurumaki-Cho}}, \city{Sinjuku-ku}, \postcode{162-0041}, \state{Tokyo}, \country{Japan}}

\abstract{
We consider an {\em enforce operator} on impartial rulesets similar to the Muller Twist and the comply/constrain operator 
of Smith and St\u anic\u a, 2002. Applied to the rulesets A and B, on each turn the opponent enforces one of the rulesets and the current player complies, by playing a move in that ruleset. If the outcome table of the enforce variation of A and B is the same as the outcome table of A, then we say that A dominates B. We find necessary and sufficient conditions for this relation. Additionally, we define a {\em selective operator} and explore a distributive-lattice-like structure within applicable rulesets. Lastly, we define nim-values under enforce-rulesets, and establish that the Sprague-Grundy theory continues to hold, along with illustrative examples.

}




\maketitle

\section{Introduction}
\label{sec:Intro}
Two players, Alice and Bob meet to play the game displayed in Figure~\ref{fig:disjunctive(knight.nim+bishop.nim)}. Last move wins, and the players alternate turns. There are two game pieces, but each individual piece is a combination of two sub-pieces that follow standard {\sc Chess} rules, while each move has to decrease the Manhattan distance to $(0,0)$.\footnote{In the two-dimensional space, the Manhattan distance between two points $(x_1, y_1)$ and $(x_2, y_2)$ is  $|x_1-x_2| +|y_1-y_2|$.}  At each stage of play, the current player decides which piece they want to move, but the opponent decides the current rules of that piece. After moving, any such enforced rule is forgotten. Pieces are invisible with respect to each other, i.e. they may reside in the same square or jump each other. 
 Suppose that Alice starts, and she wishes to play the piece $\large \symbishop \odot \large \symrook$; then Bob enforces either $\large \symbishop$ or $\large\symrook$; in either case, the other component, $
 \large \symknight \odot \large \symrook$, remains at its original position. Suppose that Bob enforces the Bishop. Then Alice might play this piece to the square $(1,0)$. At this point, say Bob wants to play the other piece, and suppose that Alice enforces the Knight. Then Bob has three options. Suppose that he plays such that the next position is as in Figure~\ref{fig:disjunctive(knight.nim+bishop.nim)2}.
\begin{figure}[htbp!]
\begin{center}
\begin{tikzpicture}[scale = 1,
         box/.style={rectangle,draw=gray!60!yellow,  thick, minimum size=1 cm}
]
\foreach \x in {0,1,...,5}{
     \foreach \y in {0,1,...,5}
         \node[box] at (\x,\y){};
 }
\draw[ultra thick, color=gray!60!yellow] (-0.5,-0.5) -- (-0.5, 5.5);
\draw[ultra thick, color=gray!60!yellow] (-0.5, 5.5)-- (5.5, 5.5);
\draw (3,1) node {\large\symbishop \!$\cdot$\!\! \symrook};
\draw[line width=0.65 pt] (3,1) circle (12.2 pt);
\draw (5,2) node {\large\symknight $\cdot$\!\! \symrook};
\draw[ line width=0.65 pt] (5,2) circle (12.2 pt);
\draw (0, 5.9) node {$0$};
\draw (-0.9, 5) node {$0$};
\end{tikzpicture}
\end{center}
 \vspace{3mm}
\caption{The technical term for this position is ``the disjunctive sum'' $\large\symbishop \odot \large\symrook + \large\symknight \odot \large\symrook$. Is this a first or second player win?}\label{fig:disjunctive(knight.nim+bishop.nim)}
\end{figure}
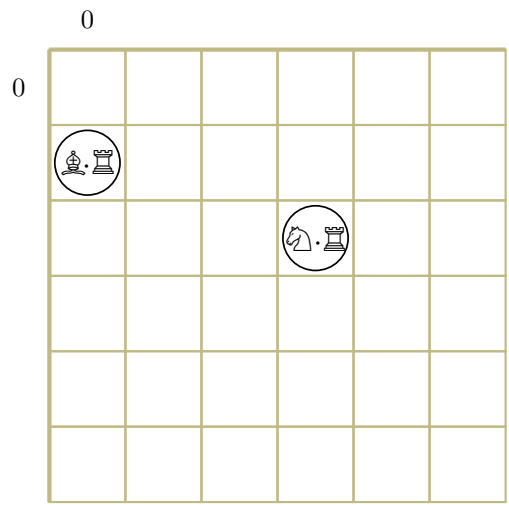
\begin{figure}[htbp!]
\begin{center}
\begin{tikzpicture}[scale = 1,
         box/.style={rectangle,draw=gray!60!yellow,  thick, minimum size=1 cm}
]

\foreach \x in {0,1,...,5}{
     \foreach \y in {0,1,...,5}
         \node[box] at (\x,\y){};
 }
\draw[ultra thick, color=gray!60!yellow] (-0.5,-0.5) -- (-0.5, 5.5);
\draw[ultra thick, color=gray!60!yellow] (-0.5, 5.5)-- (5.5, 5.5);

\draw (0,4) node {\large\symbishop \!$\cdot$\!\! \symrook};
\draw[line width=0.65 pt] (0,4) circle (12.2 pt);
\draw (3,3) node {\large\symknight $\cdot$\!\! \symrook};
\draw[ line width=0.65 pt] (3,3) circle (12.2 pt);
\draw (0, 5.9) node {$0$};
\draw (-0.9, 5) node {$0$};

\end{tikzpicture}
\end{center}
 \vspace{3mm}
\caption{The position after two moves. Alice to play. }\label{fig:disjunctive(knight.nim+bishop.nim)2}
\end{figure}
Could either player have played better? We defer the solution of this game towards the end of Section~\ref{sec:nimvalueforcomply}.
 
Here, we generalize this idea to a comply/constrain operator on pairs of rulesets that we dub the {\em enforce operator}. First,  we study play on a single component. Given a couple of rulesets with a common position set, at each stage of play, before the current player moves, exactly one of the rulesets is enforced by the opponent (meaning that all the other ones are prohibited).\footnote{Usually, there will be a pair of rulesets, but any finite (or infinite) number of rulesets could belong to the given set.} Let us give one more example, this time played on a single component.\\

\begin{example}
\label{example:Yama}

The ruleset {\sc Yama Nim} resembles {\sc Two Heap Nim}, except that a player must remove at least two tokens 
from one heap and, in the same move, add one token 
to the other heap \cite{KiY23, Y23}. 
For an example of our enforce-operator,  if the pair of rulesets is {\sc Nim} and {\sc Yama Nim}, and the position is $(1,1)$, then the current player cannot move if the other player enforces {\sc Yama Nim}. That would be a successful previous player constraint. If the position is $(0,2)$, the current player can win immediately unless {\sc Yama Nim} is enforced. By following the ruleset of {\sc Yama Nim}, the only option is $(1,0)$, from which {\sc Yama Nim} should be enforced. Thus, in the combined game of {\sc Nim} and {\sc Yama Nim}, $(0,2)$ is a win for the current player. 

We will show how {\sc Yama Nim} `dominates' {\sc Nim} in the sense that the solution of {\sc Yama Nim} coincides with the solution of the enforce-ruleset of {\sc Nim} and {\sc Yama Nim}.\\
\end{example}

We generalize this idea to a concept of {\em ruleset domination}  for classes of impartial combinatorial games. Our first main result is a precise description of when a ruleset dominates another ruleset.

In addition, we establish that domination is not an order, but we find properties close to those of an order. We define the notion of strong domination with respect to pairs of rulesets and show that strong domination implies domination. We prove that strong domination satisfies the desired transitive properties.

Further, we study another operator referred to as the {\em selective operator}; on each turn, the current player selects exactly one of a given set of rulesets and plays according to its rules. The two operators can be combined as long as the rulesets share the same position set, i.e. `game board'. We demonstrate a distributive-lattice-like structure of the combined operators. 

As mentioned in the first example (Figure~\ref{fig:disjunctive(knight.nim+bishop.nim)}) we generalize our enforce operator to disjunctive sum play.  There are several game components; at each stage of play, the current player picks exactly one component and plays in that component according to its ruleset (whether a combination component or not), while any other component remains the same.
In order to solve games in this setting, we define nim-values, a.k.a. nimbers, under enforce-rulesets and argue that the famous Sprague-Grundy theory continues to hold, which is our second main result. In particular, we prove that when there are several game components, some of them with enforce-rulesets, their disjunctive sum is a loss for the current player if and only if the nim-sum of their defined nimbers is $0$.

The paper is organized as follows. Section~\ref{sec:review} gives a literature review. Section~\ref{sec:def_impartial} defines our combinatorial games. Section~\ref{sec:complyoperator} introduces selective and enforce operators, and solves the ruleset domination problem.  Section~\ref{sec:distributive} establishes a distributive-lattice-like structure for our operators. Section~\ref{sec:nimvalueforcomply} discusses nimbers under the enforce operator.  Section~\ref{sec:examples} solves the initial problem and provides a number of more game examples, and finally, in Section~\ref{sec:openp}, we propose some open problems.

\section{Muller Twist literature}\label{sec:review}
Our enforce operator belongs to the family of so-called Muller Twist games, a.k.a. comply/constrain games, or games with a blocking maneuver. Let us briefly review this development. 
The Muller Twist of a combinatorial game was introduced/popularized through the ruleset {\sc Quarto}.~The first occurrences in the literature are \cite{HR01} and \cite{SS02}, and the term {\em comply/constrain} was coined. At each stage of play, before the current player makes a move, the opponent may constrain their set of options. The current player complies and plays according to the constrained ruleset. After each move, any previous constraint is forgotten. Various contributions study a variation of a Muller Twist as a {\em blocking maneuver}, \cite{HR01,HR1,HR,GS04,HL06,L09,L11,L15,CLN17}. In such rulesets, typically, a given parameter determines the maximum number of prohibited/constrained moves within (subsets of) the options of an underlying combinatorial game.

In the first research paper on the subject, \cite{SS02}, they study nim-type games where the opponent blocks the removal of an odd or even number of pieces, and in \cite{GS04}, they generalize this to: ``the number taken must not be equivalent to some numbers modulo $n$”. Both these papers can be viewed in the light of the enforce operator, the latter would require several combined rulesets.

 In \cite{SS02}, we find a nice description of the trailblazing ruleset {\sc Quarto}: 
``The game Quarto, created by Blaise Muller and published by Gigamic, was one of the of five Mensa Games of the Year in 1993 and has received other international awards. The sixteen game pieces show all combinations of size (short or tall), shade (light or dark), solidity (shell or filled), and shape (circle or square). Two players take turns placing pieces on a four by four board and the object is to get four in a line with the same characteristic all short, for example. Only one piece can go in a cell and, once placed, the pieces stay put. Blaise Muller’s brilliant twist is that you choose the piece that your opponent must place and they return the favor after placing it.'' Note that the idea of the opponent picking the game piece resembles our first example in Figure~\ref{fig:disjunctive(knight.nim+bishop.nim)}.

  A blocking variation of  {\sc Nim} was introduced as a problem in \cite{HR01}. This variation proceeds the same way as ordinary {\sc Nim} \cite{B02}, but before each move, the opponent may block at most one option. In \cite{HR}, the authors solve this blocking variation of {\sc Nim} on three heaps. The preprint \cite{HR1}  studies compositions of blocking impartial games and computes their nim-values.

In \cite{HL06}, the authors introduce the Muller Twist in connection with {\sc Wythoff Nim}; here, bishop-type moves can be blocked.  The authors prove that the solutions of such games are close to certain Beatty sequences. 
In \cite{L09}, the author considers {\sc Two Pile Nim } with a move-size dynamic constraint on the moves.  The author proves that the winning strategy is $\mathcal P$-equivalent to {\sc Wythoff Nim} with a blocking maneuver on the bishop-type moves. In \cite{L15}, the author considers several restrictions of the game $m$-{\sc Wythoff Nim}, which is an extension of {\sc Wythoff Nim}. The author considers a blocking maneuver on the rook-type options and finds that they are $\mathcal P$-equivalent to specific congruence restrictions on the rook-type options, and that the outcomes are exactly described by Beatty sequences. In \cite{L11}, the author considers {\sc Blocking Wythoff Nim} with a generic blocking maneuver; for a given parameter $k$, at each stage of play, at most $k-1$  options may be blocked by the opponent. The author finds the winning strategies for $k=2$ and $k=3$. 
 In \cite{CLN17}, the authors demonstrate that  {\sc Blocking Wythoff  Nim}  can be solved by a cellular automaton. The authors also present experimental results showing fascinating self-organized structures  that are invariant of $k$ as the blocking parameter $k$ increases. Another class of games in the context of  Muller Twist is {\sc Push the Button} \cite{DU18}, a game that has two rulesets, denoted as $A$ and $B$. A game starts by moving according to the $A$ ruleset, but at some point, a player may press the button, and all subsequent moves are played according to the $B$ ruleset. Our enforce rulesets are similar, except that `the button is pressed many times', switching back and forth between the rulesets. The authors of \cite{DU18} study pairwise combinations of the classical rulesets {\sc Nim}, {\sc Wythoff Nim}, and {\sc Euclid}, and compute outcomes of pairwise combined rulesets.
  
 In \cite{H08, SR18}, the authors consider games where the opponent enforces either a subtraction set $S$ or its complement. Both papers study the problem of periodicity of nimbers. This ruleset combination belongs to the enforce operator. We generalize their work by considering generic impartial rulesets.  

\section{Impartial rulesets and games}
\label{sec:def_impartial}
Our terminology and notation is an adaptation of \cite{S13} (see also \cite{C76} and \cite{BCG82}). Here, we will make a clear distinction between `ruleset' and `game'. Usually, one is interested in a comparison of game positions (or game values), but here, the main interest is a comparison of entire rulesets. This type of research will require a common `game board'.\\

\begin{definition}[Impartial Ruleset]\label{def:impartial}
  Let $\X$ be a set of game positions and let $f:\X \to 2^\X$ be an option map which sends $\x \in \X$ to its set of options $f(\x) \subset \X$. Then $(\X,f)$ is an impartial ruleset.\\ 
  
 \end{definition}

Informally, we call the set $\X$ of all (starting) positions, the `game board'. \\

\begin{definition}[Terminal Position]\label{def:terminal}
  If $\x \in \X$ satisfies $f(\x)=\emptyset$, then $\x$ is a terminal position.\\ 
\end{definition}

All our rulesets will be short and we define the same in the following.\\

\begin{definition}[Short Ruleset]\label{def:shortgame}
  The ruleset $A=(\X,f)$ is short if for any $\x \in \X$, there exists a non-negative integer $T_A(\x)$ such that every play sequence,  even ignoring the alternating play condition, starting in $\x$, terminates in at most $T_A(\x)$ moves.\\
\end{definition}

  Given a short ruleset $(\X,f)$, we play a 2-player game by assigning a starting position $\x\in \X$ and a starting player. The players choose an option alternately, and a player who must play from a terminal position loses (normal-play convention). \\

A fundamental result of short normal-play impartial rulesets (a variation of Zermelo's folklore theorem \cite{Z13}) states that we can recursively define a perfect play outcome, to every game position. These outcomes are referred to as    $\mathcal{P}$-positions  (the previous player wins) and $\mathcal{N}$-positions (the current player wins).\\

\begin{definition}[Outcome]
\label{def:outcome}
Consider a short impartial ruleset $A = (\X, f)$. Then $\mathcal{O}(\x)=\mathcal{O}_A(\x) \in \{\mathcal{P},\mathcal{N}\}$ is  the perfect play outcome of the position $\x \in \X$, under ruleset $A$. \\
\end{definition}

Let $\mathbb{Z}_{> 0}=\{1,2,\ldots\}$ and $\mathbb{Z}_{\geq 0}=\mathbb{Z}_{> 0} \cup\{0\}$ denote the positive and non-negative integers respectively. The next examples review four impartial rulesets, together with their outcomes, on the common game board  $\X=\mathbb{Z}_{\geq 0} \times \mathbb{Z}_{\geq 0}$. \\

\begin{example}[{\sc Two Heap Nim}]
    For all positions $(x,y)\in \X$ of {\sc Nim},  $\large {\symrook}$, the set of options is
    \begin{eqnarray*}
    \label{Eqn:f_function_nim}
       f((x,y)) &=& \{(x-i,y) \mid 1 \leq i \leq x \}
                          \cup\; \{(x,y-i)  \mid 1 \leq i \leq y \}, 
  \end{eqnarray*}
and $\mathcal{O}_{\large {\symrook}}((x,y))=\mathcal{P}$ if and only if $x = y$.\\
\end{example}
We introduce the ruleset {\sc White Bishop}.\\

\begin{example}[{\sc White Bishop}]
    For all positions $(x,y) \in \X$ of {\sc White Bishop},  $\large {\symbishop}$, the set of options is
    \begin{eqnarray*}
    \label{Eqn:f_functioncornerthebishop}
       f((x,y)) &=&  \{(x-i,y-i) \mid 1 \leq i \leq \min(x,y) \}, 
  \end{eqnarray*}
and $\mathcal{O}_{\large {\symbishop}}((x,y))=\mathcal{P}$ if and only if $x=0$ or $y=0$.\\ 
 \end{example}

\begin{example}[{\sc Wythoff Nim} \cite{W07}, a.k.a. {\sc Corner the Lady} \cite{GM97}]
Consider {\sc Wythoff Nim}, $\large {\symqueen}$. Then, 
    for all $(x,y) \in \X$, the set of options is
    \begin{eqnarray*}
    \label{Eqn:f_functionWythoff}
       f((x,y)) &=& \{(x-i,y) \mid 1 \leq i \leq x \} \\
                         & & \cup\; \{(x,y-i)  \mid 1 \leq i \leq y \} \\
                         & & \cup\; \{(x-i,y-i)  \mid 1 \leq i \leq \min(x,y) \}. 
  \end{eqnarray*}
 The outcome  $\mathcal{O}_{\symqueen}((x,y))=\mathcal{P}$ if and only if $x=\lfloor n\alpha\rfloor$ and $y=x+n$, for some $n\in \mathbb Z_{\ge 0}$, where $\alpha = \frac{1+\sqrt{5}}{2}$, or vice versa.\\ 
 \end{example}

\begin{example}
[{\sc Yama Nim} \cite{KiY23,Y23}]
    For all positions $(x,y) \in \X$ of {\sc Yama Nim}, $\large\rook_{\bf Y}$, the set of options is
  \begin{eqnarray*}
    \label{Eqn:y_function}
       f((x,y)) &=&\{(x-i,y+1) \mid 2 \leq i \leq x \}
                   \cup\; \{(x+1,y-i) \mid 2 \leq i \leq y \}, 
  \end{eqnarray*}
  and $\mathcal{O}_{\large\rook_{\bf Y}}((x,y))=\mathcal{P}$ if and only if $|x-y| \leq 1$.\\
  \end{example}

We illustrate the options of the rulesets $\large {\symrook}$, $\large {\symbishop}$, $\large {\symqueen}$ and  $\large\rook_{\bf Y}$ in Figure \ref{fig:sel}.    
The game pieces represents typical starting positions, $x\in \X$ and the white circles represent the option set,  $f(\x)$, of the respective ruleset. 
The initial $\mathcal{P}$-positions of these rulesets are illustrated in Figure~\ref{fig:nim,bishop,whythoff,yama}.

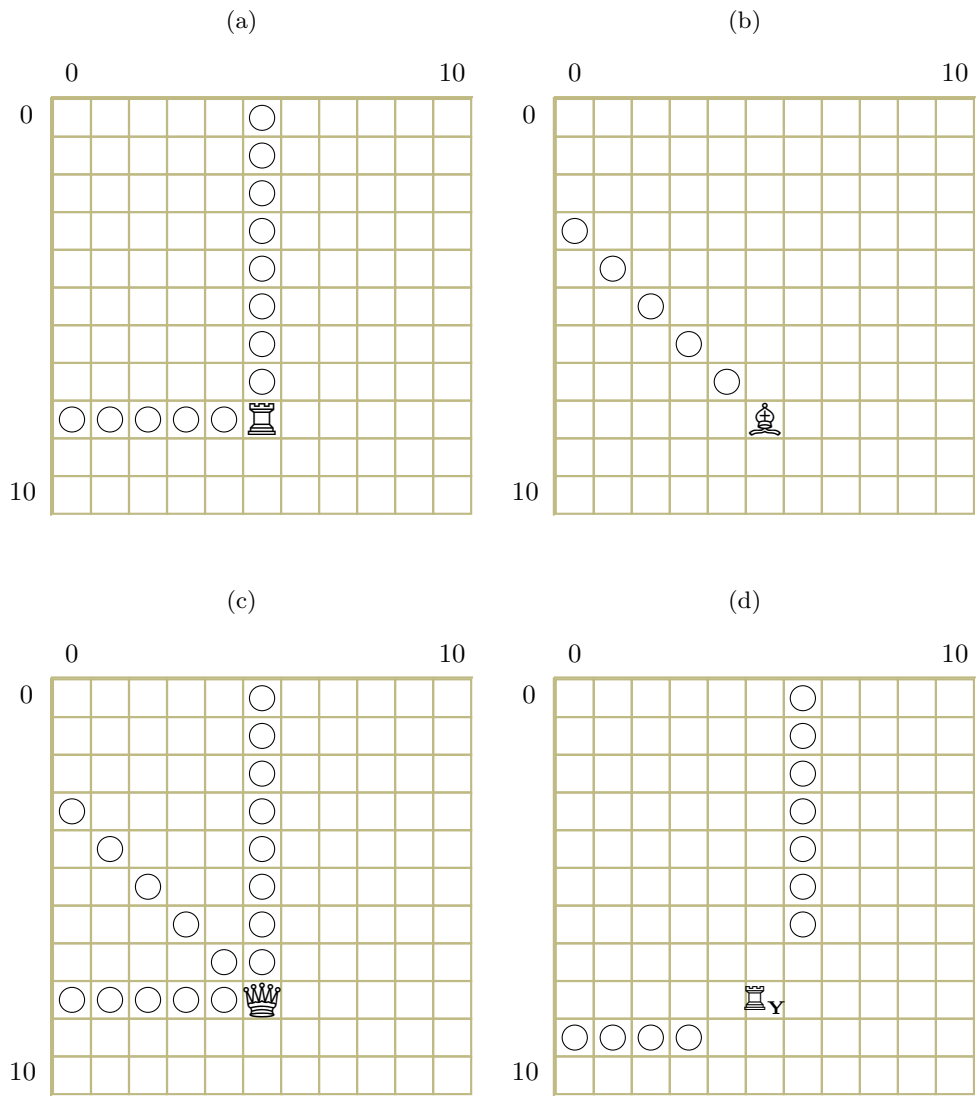
\begin{figure}[htbp!]
  \begin{minipage}{.5\textwidth}
  \subcaption{}
    \centering
    \begin{tikzpicture}[scale = 0.5,
         box/.style={rectangle,draw=gray!60!yellow,  thick, minimum size=.5 cm},
     ]

 \foreach \x in {0,1,...,10}{
     \foreach \y in {0,1,...,10}
         \node[box] at (\x,\y){};
 }

\draw[thick, color=gray!60!yellow] (-0.55,-0.55) -- (-0.55, 10.55);
\draw[thick, color=gray!60!yellow] (-0.55, 10.55)-- (10.55, 10.55);

\draw (0, 11.2) node {$0$};
\draw (10, 11.2) node {$10$};

\draw (-1.2, 10.1) node {$0$};
\draw (-1.3, 0.1) node {$10$};

\draw (5,2) node {$\Large \symrook$};
\draw (5,3) node [draw, shape = circle, fill = white, minimum size = 0.1cm] (){};
\draw (5,4) node [draw, shape = circle, fill = white, minimum size = 0.1cm] (){};
\draw (5,5) node [draw, shape = circle, fill = white, minimum size = 0.1cm] (){};
\draw (5,6) node [draw, shape = circle, fill = white, minimum size = 0.1cm] (){};
\draw (5,7) node [draw, shape = circle, fill = white, minimum size = 0.1cm] (){};
\draw (5,8) node [draw, shape = circle, fill = white, minimum size = 0.1cm] (){};
\draw (5,9) node [draw, shape = circle, fill = white, minimum size = 0.1cm] (){};
\draw (5,10) node [draw, shape = circle, fill = white, minimum size = 0.1cm] (){};
\draw (0,2) node [draw, shape = circle, fill = white, minimum size = 0.1cm] (){};
\draw (1,2) node [draw, shape = circle, fill = white, minimum size = 0.1cm] (){};
\draw (2,2) node [draw, shape = circle, fill = white, minimum size = 0.1cm] (){};
\draw (3,2) node [draw, shape = circle, fill = white, minimum size = 0.1cm] (){};
\draw (4,2) node [draw, shape = circle, fill = white, minimum size = 0.1cm] (){};

\end{tikzpicture}
    \label{fig:nim}
  \end{minipage}
  \begin{minipage}{.5\textwidth}
  \subcaption{}
    \centering
    \begin{tikzpicture}[scale = 0.5,
         box/.style={rectangle,draw=gray!60!yellow, thick, minimum size=.5 cm},
     ]

 \foreach \x in {0,1,...,10}{
     \foreach \y in {0,1,...,10}
         \node[box] at (\x,\y){};
 }

\draw[thick, color=gray!60!yellow] (-0.55,-0.55) -- (-0.55, 10.55);
\draw[thick, color=gray!60!yellow] (-0.55, 10.55)-- (10.55, 10.55);

\draw (0, 11.2) node {$0$};
\draw (10, 11.2) node {$10$};

\draw (-1.2, 10.1) node {$0$};
\draw (-1.3, 0.1) node {$10$};

\draw (5,2) node {$\Large\symbishop$};
\draw (4,3) node [draw, shape = circle, fill = white, minimum size = 0.1cm] (){};
\draw (3,4) node [draw, shape = circle, fill = white, minimum size = 0.1cm] (){};
\draw (2,5) node [draw, shape = circle, fill = white, minimum size = 0.1cm] (){};
\draw (1,6) node [draw, shape = circle, fill = white, minimum size = 0.1cm] (){};
\draw (0,7) node [draw, shape = circle, fill = white, minimum size = 0.1cm] (){};
\end{tikzpicture}
\label{fig:ctb}
\end{minipage}

  \begin{minipage}{.5\textwidth}
  \vspace{7 mm}
  \subcaption{}
    \centering
  \begin{tikzpicture}[scale = 0.5,
         box/.style={rectangle,draw=gray!60!yellow,  thick, minimum size=.5 cm},
     ]

\foreach \x in {0,1,...,10}{
     \foreach \y in {0,1,...,10}
         \node[box] at (\x,\y){};
 }

\draw[thick, color=gray!60!yellow] (-0.55,-0.55) -- (-0.55, 10.55);
\draw[thick, color=gray!60!yellow] (-0.55, 10.55)-- (10.55, 10.55);

\draw (0, 11.2) node {$0$};
\draw (10, 11.2) node {$10$};

\draw (-1.2, 10.1) node {$0$};
\draw (-1.3, 0.1) node {$10$};

\draw (5,2) node {$\Large\symqueen$};
\draw (5,3) node [draw, shape = circle, fill = white, minimum size = 0.1cm] (){};
\draw (5,4) node [draw, shape = circle, fill = white, minimum size = 0.1cm] (){};
\draw (5,5) node [draw, shape = circle, fill = white, minimum size = 0.1cm] (){};
\draw (5,6) node [draw, shape = circle, fill = white, minimum size = 0.1cm] (){};
\draw (5,7) node [draw, shape = circle, fill = white, minimum size = 0.1cm] (){};
\draw (5,8) node [draw, shape = circle, fill = white, minimum size = 0.1cm] (){};
\draw (5,9) node [draw, shape = circle, fill = white, minimum size = 0.1cm] (){};
\draw (5,10) node [draw, shape = circle, fill = white, minimum size = 0.1cm] (){};
\draw (0,2) node [draw, shape = circle, fill = white, minimum size = 0.1cm] (){};
\draw (1,2) node [draw, shape = circle, fill = white, minimum size = 0.1cm] (){};
\draw (2,2) node [draw, shape = circle, fill = white, minimum size = 0.1cm] (){};
\draw (3,2) node [draw, shape = circle, fill = white, minimum size = 0.1cm] (){};
\draw (4,2) node [draw, shape = circle, fill = white, minimum size = 0.1cm] (){};

\draw (4,3) node [draw, shape = circle, fill = white, minimum size = 0.1cm] (){};
\draw (3,4) node [draw, shape = circle, fill = white, minimum size = 0.1cm] (){};
\draw (2,5) node [draw, shape = circle, fill = white, minimum size = 0.1cm] (){};
\draw (1,6) node [draw, shape = circle, fill = white, minimum size = 0.1cm] (){};
\draw (0,7) node [draw, shape = circle, fill = white, minimum size = 0.1cm] (){};
\end{tikzpicture}
    \label{fig:wyt}
  \end{minipage}
  \begin{minipage}{.5\textwidth}
  \vspace{7 mm}
  \subcaption{}
    \centering
    \begin{tikzpicture}[scale = 0.5,
         box/.style={rectangle,draw=gray!60!yellow,  thick, minimum size=.5 cm},
     ]
 \foreach \x in {0,1,...,10}{
     \foreach \y in {0,1,...,10}
         \node[box] at (\x,\y){};
 }

\draw[thick, color=gray!60!yellow] (-0.55,-0.55) -- (-0.55, 10.55);
\draw[thick, color=gray!60!yellow] (-0.55, 10.55)-- (10.55, 10.55);

\draw (0, 11.2) node {$0$};
\draw (10, 11.2) node {$10$};

\draw (-1.2, 10.1) node {$0$};
\draw (-1.3, 0.1) node {$10$};

\draw (5,2) node {$\large\rook_{\bf Y}$};
\draw (6,4) node [draw, shape = circle, fill = white, minimum size = 0.1cm] (){};
\draw (6,5) node [draw, shape = circle, fill = white, minimum size = 0.1cm] (){};
\draw (6,6) node [draw, shape = circle, fill = white, minimum size = 0.1cm] (){};
\draw (6,7) node [draw, shape = circle, fill = white, minimum size = 0.1cm] (){};
\draw (6,8) node [draw, shape = circle, fill = white, minimum size = 0.1cm] (){};
\draw (6,9) node [draw, shape = circle, fill = white, minimum size = 0.1cm] (){};
\draw (6,10) node [draw, shape = circle, fill = white, minimum size = 0.1cm] (){};
\draw (0,1) node [draw, shape = circle, fill = white, minimum size = 0.1cm] (){};
\draw (1,1) node [draw, shape = circle, fill = white, minimum size = 0.1cm] (){};
\draw (2,1) node [draw, shape = circle, fill = white, minimum size = 0.1cm] (){};
\draw (3,1) node [draw, shape = circle, fill = white, minimum size = 0.1cm] (){};

\end{tikzpicture}
\label{fig:yama}
\end{minipage}
 \vspace{3mm}
 \caption{The pictures represent typical options of 
 (a)~{\sc Two Heap Nim}, (b)~{\sc White Bishop}, (c)~{\sc Wythoff Nim} and (d)~{\sc Yama Nim}.
  }
   \label{fig:sel}
   \end{figure}
 \clearpage

\begin{figure}[htbp!]
 \begin{minipage}{.5\textwidth}
  \subcaption{}
    \centering
\begin{tikzpicture}[scale =0.5, box/.style={rectangle,draw=gray!60!yellow,  thick, minimum size=.5 cm},]

\foreach \x in {0,1,...,10}{
     \foreach \y in {0,1,...,10}
         \node[box] at (\x,\y){};
 }

\draw[thick, color=gray!60!yellow] (-0.55,-0.55) -- (-0.55, 10.55);
\draw[thick, color=gray!60!yellow] (-0.55, 10.55)-- (10.55, 10.55);

\draw (0, 11.2) node {$0$};
\draw (10, 11.2) node {$10$};

\draw (-1.2, 10.1) node {$0$};
\draw (-1.3, 0.1) node {$10$};

\draw (0,10) node {$\mathcal{P}$};
\draw (1,9) node {$\mathcal{P}$};
\draw (2,8) node {$\mathcal{P}$};
\draw (3,7) node {$\mathcal{P}$};
\draw (4,6) node {$\mathcal{P}$};
\draw (5,5) node {$\mathcal{P}$};
\draw (6,4) node {$\mathcal{P}$};
\draw (7,3) node {$\mathcal{P}$};
\draw (8,2) node {$\mathcal{P}$};
\draw (9,1) node {$\mathcal{P}$};
\draw (10,0) node {$\mathcal{P}$};
\end{tikzpicture}
\label{fig:nimP}
\end{minipage}
\begin{minipage}{.5\textwidth}
     \subcaption{}
    \centering
    \begin{tikzpicture}[scale = 0.5,
         box/.style={rectangle,draw=gray!60!yellow,  thick, minimum size=.5 cm},
     ]
 \foreach \x in {0,1,...,10}{
     \foreach \y in {0,1,...,10}
         \node[box] at (\x,\y){};
 }

\draw[thick, color=gray!60!yellow] (-0.55,-0.55) -- (-0.55, 10.55);
\draw[thick, color=gray!60!yellow] (-0.55, 10.55)-- (10.55, 10.55);

\draw (0, 11.2) node {$0$};
\draw (10, 11.2) node {$10$};

\draw (-1.2, 10.1) node {$0$};
\draw (-1.3, 0.1) node {$10$};

\draw (0,0) node {$\mathcal{P}$};
\draw (0,1) node {$\mathcal{P}$};
\draw (0,2) node {$\mathcal{P}$};
\draw (0,3) node {$\mathcal{P}$};
\draw (0,4) node {$\mathcal{P}$};
\draw (0,5) node {$\mathcal{P}$};
\draw (0,6) node {$\mathcal{P}$};
\draw (0,7) node {$\mathcal{P}$};
\draw (0,8) node {$\mathcal{P}$};
\draw (0,9) node {$\mathcal{P}$};
\draw (0,10) node {$\mathcal{P}$};

\draw (1,10) node {$\mathcal{P}$};
\draw (2,10) node {$\mathcal{P}$};
\draw (3,10) node {$\mathcal{P}$};
\draw (4,10) node {$\mathcal{P}$};
\draw (5,10) node {$\mathcal{P}$};
\draw (6,10) node {$\mathcal{P}$};
\draw (7,10) node {$\mathcal{P}$};
\draw (8,10) node {$\mathcal{P}$};
\draw (9,10) node {$\mathcal{P}$};
\draw (10,10) node {$\mathcal{P}$};
\end{tikzpicture}
\label{fig:bishopP}
\end{minipage}

 \begin{minipage}{.5\textwidth}
\vspace{7 mm}
  \subcaption{}
    \centering
\begin{tikzpicture}[scale = 0.5,
         box/.style={rectangle,draw=gray!60!yellow, thick, minimum size=.5 cm},
     ]

 \foreach \x in {0,1,...,10}{
     \foreach \y in {0,1,...,10}
         \node[box] at (\x,\y){};
 }

\draw[thick, color=gray!60!yellow] (-0.55,-0.55) -- (-0.55, 10.55);
\draw[thick, color=gray!60!yellow] (-0.55, 10.55)-- (10.55, 10.55);

\draw (0, 11.2) node {$0$};
\draw (10, 11.2) node {$10$};

\draw (-1.2, 10.1) node {$0$};
\draw (-1.3, 0.1) node {$10$};

\draw (0,10) node {$\mathcal{P}$};
\draw (2,9) node {$\mathcal{P}$};
\draw (1,8) node {$\mathcal{P}$};
\draw (5,7) node {$\mathcal{P}$};
\draw (7,6) node {$\mathcal{P}$};
\draw (3,5) node {$\mathcal{P}$};
\draw (10,4) node {$\mathcal{P}$};
\draw (4,3) node {$\mathcal{P}$};
\draw (6,0) node {$\mathcal{P}$};

\end{tikzpicture}
\label{fig:wythP}
\end{minipage}
\begin{minipage}{.5\textwidth}
\vspace{7 mm}
 \subcaption{}
    \centering
    \begin{tikzpicture}[scale = 0.5,
         box/.style={rectangle,draw=gray!60!yellow,  thick, minimum size=.5 cm},
     ]
\foreach \x in {0,1,...,10}{
     \foreach \y in {0,1,...,10}
         \node[box] at (\x,\y){};
 }

\draw[thick, color=gray!60!yellow] (-0.55,-0.55) -- (-0.55, 10.55);
\draw[thick, color=gray!60!yellow] (-0.55, 10.55)-- (10.55, 10.55);

\draw (0, 11.2) node {$0$};
\draw (10, 11.2) node {$10$};

\draw (-1.2, 10.1) node {$0$};
\draw (-1.3, 0.1) node {$10$};

\draw (0,10) node {$\mathcal{P}$};
\draw (1,10) node {$\mathcal{P}$};

\draw (1,9) node {$\mathcal{P}$};
\draw (0,9) node {$\mathcal{P}$};
\draw (2,9) node {$\mathcal{P}$};

\draw (2,8) node {$\mathcal{P}$};
\draw (1,8) node {$\mathcal{P}$};
\draw (3,8) node {$\mathcal{P}$};

\draw (3,7) node {$\mathcal{P}$};
\draw (2,7) node {$\mathcal{P}$};
\draw (4,7) node {$\mathcal{P}$};

\draw (4,6) node {$\mathcal{P}$};
\draw (3,6) node {$\mathcal{P}$};
\draw (5,6) node {$\mathcal{P}$};

\draw (5,5) node {$\mathcal{P}$};
\draw (4,5) node {$\mathcal{P}$};
\draw (6,5) node {$\mathcal{P}$};

\draw (6,4) node {$\mathcal{P}$};
\draw (5,4) node {$\mathcal{P}$};
\draw (7,4) node {$\mathcal{P}$};

\draw (7,3) node {$\mathcal{P}$};
\draw (6,3) node {$\mathcal{P}$};
\draw (8,3) node {$\mathcal{P}$};

\draw (8,2) node {$\mathcal{P}$};
\draw (7,2) node {$\mathcal{P}$};
\draw (9,2) node {$\mathcal{P}$};

\draw (9,1) node {$\mathcal{P}$};
\draw (8,1) node {$\mathcal{P}$};
\draw (10,1) node {$\mathcal{P}$};

\draw (10,0) node {$\mathcal{P}$};
\draw (9,0) node {$\mathcal{P}$};
\end{tikzpicture}
\label{fig:yamaP}
\end{minipage}

\vspace{3mm}
\caption{The pictures represent  initial $\mathcal{P}$-positions of the rulesets in Figure~\ref{fig:nim,bishop,whythoff,yama},  (a)~$\large {\symrook}$, (b)~$\large {\symbishop}$, (c)~$\large {\symqueen}$ and (d)~$\large\rook_{\bf Y}$.}
\label{fig:nim,bishop,whythoff,yama}
\end{figure}
\clearpage
\section{New operators of combinatorial games}
\label{sec:complyoperator}
In this section, we introduce the {\em selective} and {\em enforce} 
operators. The selective operator (Definition~\ref{def:seloperator}) is the natural `dual' to our defined enforce operator (Definition~\ref{def:enfop}). Together they satisfy some nice properties (see Section~\ref{sec:distributive}). We study short rulesets and in view of our enforce operator we need to generalize this notion.\\

\begin{definition}[Jointly Short]\label{def:JShort}
Let  $A=(\X, f_A)$ and $B=(\X, f_B)$ be short rulesets on a set $\X$.    
Then $A$ and $B$ are  jointly short if the ruleset $(\X, f)$ given by, for all $\x\in \X$, $f(\x)=f_A(\x)\cup f_B(\x)$,  is short.\\ 
\end{definition}

Informally we will say $A$-options or $B$-options for these respective sets $f_A(\x)$ or $f_B(\x)$. 
In more generality, a set of short rulesets $\{A_i\}$ is {\em jointly short}, if the ruleset $(\X, f)$ given by, for all $\x\in \X$, $f(\x)=\bigcup_i f_{A_i}(\x)$,  is short.\\

\begin{definition}[Selective Operator]
\label{def:seloperator}
  Let $A$ and $B$ be jointly short rulesets on 
  a set $\X$. A selective operator on these rulesets is the ruleset $A \circledcirc B$, where,  on each turn,  
  the current player selects one of the two rulesets $A$ or $B$ and plays according to its rules.\\ 
\end{definition}
For all $\x \in \X$, we may interpret $f(\x)=f_A(\x)\cup f_B(\x)$ as the options of $A \circledcirc B$. Note that the selective operator is defined only if the rulesets are jointly short. Note that if we relax the condition that $A$ and $B$ be jointly short, even if they were individually short, the ruleset $A \circledcirc B$ might not be defined. For example, if $\X = \{(x, y)\mid x, y \in \mathbb{Z}_{\geq 0}\}$,  with $$f_A((x,y)) = \{(x', y') \mid x', y' \in \mathbb{Z}_{\geq 0}, x' < x\}, \mbox{ and } $$ 
$$f_B((x,y)) = \{(x', y') \mid x', y' \in \mathbb{Z}_{\geq 0}, y' < y\},$$ then $A$ and $B$ are short, but $A$ and $B$ are not jointly short.\\

{\sc Wythoff Nim} $\large {\symqueen}$ is a good example of a selective operator; combine $\large {\symrook}$ with $\large {\symbishop}$. See  Figure \ref{fig:sel}.\\

\begin{example}
\label{exp:selsum}
Consider the  rulesets $\large\rook$ and  $\large\rook_{\bf Y}$. Then, the options of $\large\rook \circledcirc \large\rook_{\bf Y}$ are as in Figure~\ref{fig:selective_operator}. 
The white circles represent the moves of  $\large\rook$, and the rectangles represent the moves of $\large\rook_{\bf Y}$. The current player can choose either of the rulesets, so all of them are the moves of $\large\rook \circledcirc \large\rook_{\bf Y}$.\\
\end{example}

\begin{figure}[htbp!]
\begin{center}
\begin{tikzpicture}[scale = 0.5,
         box/.style={rectangle,draw=gray!60!yellow,  thick, minimum size=.5 cm},
     ]

 \foreach \x in {0,1,...,10}{
     \foreach \y in {0,1,...,10}
         \node[box] at (\x,\y){};
 }


\draw[thick, color=gray!60!yellow] (-0.55,-0.55) -- (-0.55, 10.55);
\draw[thick, color=gray!60!yellow] (-0.55, 10.55)-- (10.55, 10.55);

\draw (0, 11.2) node {$0$};
\draw (10, 11.2) node {$10$};

\draw (-1.2, 10.1) node {$0$};
\draw (-1.3, 0.1) node {$10$};

\draw (5,2) node [draw, shape = circle, fill = black, minimum size = 0.1cm] (){};
\draw (5,3) node [draw, shape = circle, fill = white, minimum size = 0.1cm] (){};
\draw (5,4) node [draw, shape = circle, fill = white, minimum size = 0.1cm] (){};
\draw (5,5) node [draw, shape = circle, fill = white, minimum size = 0.1cm] (){};
\draw (5,6) node [draw, shape = circle, fill = white, minimum size = 0.1cm] (){};
\draw (5,7) node [draw, shape = circle, fill = white, minimum size = 0.1cm] (){};
\draw (5,8) node [draw, shape = circle, fill = white, minimum size = 0.1cm] (){};
\draw (5,9) node [draw, shape = circle, fill = white, minimum size = 0.1cm] (){};
\draw (5,10) node [draw, shape = circle, fill = white, minimum size = 0.1cm] (){};
\draw (0,2) node [draw, shape = circle, fill = white, minimum size = 0.1cm] (){};
\draw (1,2) node [draw, shape = circle, fill = white, minimum size = 0.1cm] (){};
\draw (2,2) node [draw, shape = circle, fill = white, minimum size = 0.1cm] (){};
\draw (3,2) node [draw, shape = circle, fill = white, minimum size = 0.1cm] (){};
\draw (4,2) node [draw, shape = circle, fill = white, minimum size = 0.1cm] (){};

\draw (6,4) node [draw, shape = rectangle, fill = white, minimum size = 0.1cm] (){};
\draw (6,5) node [draw, shape = rectangle, fill = white, minimum size = 0.1cm] (){};
\draw (6,6) node [draw, shape = rectangle, fill = white, minimum size = 0.1cm] (){};
\draw (6,7) node [draw, shape = rectangle, fill = white, minimum size = 0.1cm] (){};
\draw (6,8) node [draw, shape =rectangle, fill = white, minimum size = 0.1cm] (){};
\draw (6,9) node [draw, shape = rectangle, fill = white, minimum size = 0.1cm] (){};
\draw (6,10) node [draw, shape = rectangle, fill = white, minimum size = 0.1cm] (){};
\draw (0,1) node [draw, shape = rectangle, fill = white, minimum size = 0.1cm] (){};
\draw (1,1) node [draw, shape = rectangle, fill = white, minimum size = 0.1cm] (){};
\draw (2,1) node [draw, shape = rectangle, fill = white, minimum size = 0.1cm] (){};
\draw (3,1) node [draw, shape = rectangle, fill = white, minimum size = 0.1cm] (){};
\end{tikzpicture}
\end{center}
\vspace{3mm}
\caption{An example of the selective operator  as in Example~\ref{exp:selsum}, where the filled circle is the game piece $\large\symrook\circledcirc \large\symrook_{\bf Y}$.}\label{fig:selective_operator}
\end{figure}
Notice that for both $(A\circledcirc B)\circledcirc C$ and $A\circledcirc(B\circledcirc C)$, the current player selects one of $A$, $B$ and $C$, so the associative law $(A\circledcirc B)\circledcirc C=A\circledcirc(B\circledcirc C)$ holds.\\
Impartial game values are all incomparable,  with respect to the standard partial order of games. Here, we are concerned with entire rulesets and their comparisons. We will discuss the problem of ruleset domination with respect to our enforce operator. Intuitively, one ruleset dominates the other if, during play,  the second ruleset can always be blocked by the players without changing the outcome. As a consequence, the outcomes of the combined game will be the same as the outcomes of the first game.\\ 

\begin{definition}[Enforce Operator]
\label{def:enfop}
  Let $A$ and $B$ be jointly short rulesets on a set $\X$. 
  
  The enforce operator on these rulesets is the combined ruleset $A \odot B$, where 
  on each turn, before a play, the opponent enforces one of these two rulesets, and the current player plays according to that ruleset.   
That is, for all $\x \in \X$, the opponent enforces either of the set of options $f_A(\x)$ or $f_B(\x)$.\\
\end{definition}

Observe that the enforce operator belongs to the family of comply/constrain, so every enforce maneuver is forgotten after each move. 
Notice that for both $(A\odot B)\odot C$ and $A\odot(B\odot C)$, the opponent enforces one of $A$, $B$ and $C$, so the associative law $(A\odot B)\odot C=A\odot(B\odot C)$ holds.\\

\begin{definition}[Ruleset Domination]
\label{def:domination_confused_simi}
Consider two rulesets $A$ and $B$ on a set $\X$. 
If, for any $\x \in \X$, $\mathcal{O}_{A \odot B}(\x) = \mathcal{O}_A(\x)$, then ruleset $A$ dominates ruleset $B$, denoted $A\vdash B$.
If both 
$B\vdash A$ and $A\vdash B$, then the rulesets $A$ and $B$ are similar, denoted  $A \simeq B$. If neither domination holds, then the rulesets $A$ and $B$ are confused, denoted $A \parallel B$.\\

\end{definition}

\begin{example}[Dominated Ruleset]
Consider the  rulesets $\large\rook$ and   $\large\bishop$. For all $\x$, $\mathcal{O}_{\large\rook \odot \large\bishop}(\x)=\mathcal{O}_{\large\bishop}(\x)$, and hence   $\large\bishop$ dominates $\large\rook$. Namely, if the game starts at the edge of the game board, that is, if one of the coordinates is $0$, then the previous player enforces $\large\bishop$, and wins. Otherwise, the current player has a good option in both rulesets. Namely, they play to the edge of the game board   and enforce  $\large\bishop$.\\

\end{example}

\begin{example}[Confused Rulesets]\label{Ex:confused}
Consider the rulesets  $\large\rook_{\bf Y}$ and   $\large\bishop$. The initial $\mathcal P$-positions of $\large\rook_{\bf Y}\odot \large\bishop$ are depicted in Figure~\ref{tab:Yama_Corner_outcome}. 
 Suppose the starting position is $\x=(10,10)$. From Figure~\ref{fig:nim,bishop,whythoff,yama} (d) and Figure~\ref{tab:Yama_Corner_outcome} we have $  \mathcal{O}_{\large\rook_{\bf Y}}(\x) =\mathcal{P}$ and $\mathcal{O}_{\large\rook_{\bf Y} \odot \large\bishop}(\x)=\mathcal{N}$, respectively. Hence $\large\rook_{\bf Y}\not\,\vdash \large\bishop$. Consider the starting position $\x=(1,1)$.  From  Figure~\ref{fig:nim,bishop,whythoff,yama} (b) and Figure~\ref{tab:Yama_Corner_outcome} we have $  \mathcal{O}_{\large\bishop}(\x) =\mathcal{N}$ and $\mathcal{O}_{\large\rook_{\bf Y} \odot \large\bishop}(\x)=\mathcal{P}$, respectively, and hence $\large\bishop \not\,\vdash A$. Therefore $\large\rook_{\bf Y} \parallel \large\bishop$;  the rulesets are confused. 

\end{example}

\begin{figure}[htbp!]
\begin{center}
\begin{tikzpicture}[scale = 0.5,
         box/.style={rectangle,draw=gray!60!yellow,  thick, minimum size=.5 cm},
     ]

\foreach \x in {0,1,...,10}{
     \foreach \y in {0,1,...,10}
         \node[box] at (\x,\y){};
 }

\draw[thick, color=gray!60!yellow] (-0.55,-0.5) -- (-0.55, 10.55);
\draw[thick, color=gray!60!yellow] (-0.55, 10.55)-- (10.55, 10.55);

\draw (0, 11.2) node {$0$};
\draw (10, 11.2) node {$10$};

\draw (-1.2, 10.1) node {$0$};
\draw (-1.3, 0.1) node {$10$};

\draw (0,0) node {$\mathcal{P}$};
\draw (0,1) node {$\mathcal{P}$};
\draw (0,2) node {$\mathcal{P}$};
\draw (0,3) node {$\mathcal{P}$};
\draw (0,4) node {$\mathcal{P}$};
\draw (0,5) node {$\mathcal{P}$};
\draw (0,6) node {$\mathcal{P}$};
\draw (0,7) node {$\mathcal{P}$};
\draw (0,8) node {$\mathcal{P}$};
\draw (0,9) node {$\mathcal{P}$};
\draw (0,10) node {$\mathcal{P}$};

\draw (1,10) node {$\mathcal{P}$};
\draw (2,10) node {$\mathcal{P}$};
\draw (3,10) node {$\mathcal{P}$};
\draw (4,10) node {$\mathcal{P}$};
\draw (5,10) node {$\mathcal{P}$};
\draw (6,10) node {$\mathcal{P}$};
\draw (7,10) node {$\mathcal{P}$};
\draw (8,10) node {$\mathcal{P}$};
\draw (9,10) node {$\mathcal{P}$};
\draw (10,10) node {$\mathcal{P}$};

\draw (1,9) node {$\mathcal{P}$};
\end{tikzpicture}
\end{center}
\vspace{3mm}
\caption{The initial $\mathcal{P}$-positions of $\large\rook_{\bf Y} \odot \large\bishop$ as in Example~\ref{Ex:confused}. }\label{tab:Yama_Corner_outcome}
\end{figure}

\subsection{A solution of the ruleset domination problem}
 We now characterize the ruleset domination problem as in Definition~\ref{def:domination_confused_simi}, i.e., what is the necessary and sufficient condition for a ruleset $A$ to dominate a ruleset $B$?\\
\begin{definition}[Property~1]
\label{def:property1}
 The ordered pair of rulesets $(A,B)$ satisfies Property~1 if, for any $\x \in \X$, if $\mathcal{O}_A(\x) = \mathcal{N}$, then there is an option $\x' \in f_B(\x)$ such that $\mathcal{O}_A(\x') = \mathcal{P}$.\\
\end{definition}
We write $P_1(A,B)$ if $(A,B)$ satisfies Property~1. The intuition of this property, with respect to the enforce operator, is as follows: since the opponent is the enforcer, they need never enforce ruleset $B$, if $P_1(A,B)$.

The ordered pair of rulesets $(A,B)$ does not satisfy Property~1 if there exists an $\x \in \X$ such that $\mathcal{O}_A(\x) = \mathcal{N}$ and, for all $\x' \in f_B(\x)$, $\mathcal{O}_A(\x') = \mathcal{N}$. We write $P_1(A,B)^c$ if $(A,B)$ does not satisfy Property~1. To continue the intuition from the previous paragraph, in this case, the opponent could benefit by enforcing ruleset $B$.\\

\begin{theorem}[Dominated Ruleset]
\label{thm:dominated_ruleset}
Assume that $A$ and $B$ are jointly short rulesets on a set $\X$.
  Then $A \vdash B$ if and only if $P_1(A,B)$ holds.
\end{theorem}

\begin{proof}
  First, assume that $(A,B)$ satisfies Property~1.
  We need to demonstrate that $A \vdash B$.
  Assume that $\mathcal{O}_A(\x) = \mathcal{N}$. Then, under ruleset $A \odot B$, either way, the current player can move to a position $\x'$ such that $\mathcal{O}_A(\x') = \mathcal{P}$. This is because if ruleset $A$ is enforced, as $\mathcal{O}_A(\x) = \mathcal{N}$, there is a position $\x' \in f_A(\x)$ such that $\mathcal{O}_A(\x') = \mathcal{P}$, and if ruleset $B$ is enforced, Property~1 guarantees there is an $\x' \in f_B(\x)$ such that $\mathcal{O}_A(\x') = \mathcal{P}$. Thus, by induction, $\mathcal{O}_{A \odot B}(\x)=\mathcal{N}$.

  On the other hand, if $\mathcal{O}_A(\x) = \mathcal{P}$, it suffices for the previous player to enforce ruleset $A$. 
 
  Namely, for any $\x' \in f_A(\x)$,  $\mathcal{O}_A(\x') = \mathcal{N}$. Thus, by induction, $\mathcal{O}_{A \odot B}(\x)=~\mathcal{P}$.

  Hence, for all $\x$, $\mathcal{O}_A(\x)=\mathcal{O}_{A\odot B}(\x)$.

  Next, assume that $A \vdash B$, and $(A,B)$ does not satisfy Property~1.  
  By domination, and $P_1(A,B)^c$, there is a position $\x \in \X$, with $\mathcal{O}_A(\x)=\mathcal{O}_{A\odot B}(\x) = \mathcal N$, such that for any $\x' \in f_B(\x)$, $\mathcal{O}_A(\x') =\mathcal{O}_{A \odot B}(\x')= \mathcal{N}$. Thus, from this $\x$, the opponent can enforce ruleset $B$, and then the current player has to move to some $\x' \in f_B(\x)$, for which $\mathcal{O}_A(\x') = \mathcal{O}_{A \odot B}(\x') = \mathcal{N}$. Thus, the current player cannot play from the $\mathcal{N}$-position $\x$ in $A \odot B$ to a $\mathcal{P}$-position $\x'$ in $A \odot B$, a contradiction.
\end{proof}

\begin{remark}
   
By virtue of Theorem~\ref{thm:dominated_ruleset}, this is a ``survival of the weakest'', since the dominated ruleset is the stronger ruleset, and it will not be enforced by the other player. The outcomes of a stronger ruleset are here irrelevant. Hence the weaker ruleset dominates the stronger ruleset.\\
\end{remark}

\begin{corollary}
\label{Cor:domination}
Let  $A$ and $B$ be jointly short rulesets on  $\X$.  
\begin{enumerate}
\item[(1)] If $f_A(\x)\subseteq f_B(\x)$, for all $\x\in \X$, then  $A\vdash B$.

\item[(2)] If, for any $\x\in \X$ with $\mathcal{O}_B(\x)=\mathcal{P}$, $\x$ is also a $\mathcal{P}$-position for ruleset $A$, then $A \vdash B$. 

\item[(3)] $A \simeq B$ if and only if $\mathcal{O}_A(\x)=\mathcal{O}_B(\x)$.
\end{enumerate}
\end{corollary}
\begin{proof} (1) By Theorem~\ref{thm:dominated_ruleset}, it suffices to find a $B$-option $\x'$ such that $\mathcal{O}_A(\x')=\mathcal{P}$, whenever $\mathcal{O}_A(\x)=\mathcal{N}$. Indeed, the existence follows from $f_A(\x)\subseteq f_B(\x)$.\\
(2) For any $\x \in \X$ with $\mathcal{O}_{A}(\x)=\mathcal{N}$, by our assumption, we have $\mathcal{O}_{B}(\x)=\mathcal{N}$; hence there is some option $\x'\in f_B(\x)$ with $\mathcal{O}_B(\x')=\mathcal{P}$. This implies $\mathcal{O}_A(\x')=\mathcal{P}$; hence  $P_1(A, B)$ holds.\\
(3) When $A \vdash B$ and $B\vdash A$, for any $\x\in \X$, we have $\mathcal{O}_A(\x)=\mathcal{O}_{A\odot B}(\x)=\mathcal{O}_B(\x)$, and hence $\mathcal{O}_A(\x)=\mathcal{O}_B(\x)$.  Conversely if $\mathcal{O}_A(\x)=\mathcal{O}_B(\x)$, then by (2), $A \vdash B$ and $B \vdash A$; hence $A\simeq B$.
\end{proof}

\begin{example}
  Consider two rulesets  $\rook_{\bf Y}$ and   $\rook$. If $\mathcal{O}_{\rook_{\bf Y}}((x, y)) = \mathcal{N}$, we may assume $x-y \geq 2$, and one can move to $(x-(x - y), y)=(y, y) \in f_{\rook}((x,y))$  which is a  $\mathcal{P}$-position of ruleset $\rook_{\bf Y}$. Hence, $P_1\left(\rook_{\bf Y}, \rook\right)$ holds, $\rook_{\bf Y} \vdash \rook$, and $\mathcal{O}_{\rook_{\bf Y} \odot \rook}((x, y))=\mathcal{O}_{\rook_{\bf Y}}((x, y))$ for any $(x,y)$. The ruleset $\rook_{\bf Y}$ also dominates  $\queen$, by the same move.\\
\end{example}

\subsection{No Transitivity}
We find a contradiction in the transitivity of ruleset domination.\\

\begin{proposition}
\label{prop:rulesetdomination}
 The binary ruleset relation of domination 
 does not satisfy the transitive law.
\end{proposition}
\begin{proof}
 Let $\X=\mathbb{Z}_{\geq 0}$ and consider the following rulesets.\\
 
    Ruleset $A$: $f_A(x)=\left\{
  \begin{array}{ll}
    \emptyset, & \text{if } x =0,\\
     \{x-1\},& \text{otherwise}.
  \end{array}
  \right. $\\
    
Ruleset $B$: $f_B(x)=\{x-(2i+1) \mid i \in \mathbb{Z}, 0 \leq i \leq \frac{x-1}{2} \}$. \\
  
Ruleset $C$: $f_C(x)=\{x-2i \mid i \in \mathbb{Z}, 1 \leq i \leq \frac{x}{2} \}$.\\

That is, for ruleset $B$, any odd number less than or equal to $x$ can be subtracted, and for ruleset $C$, any even number less than or equal to $x$ can be subtracted. 

\begin{table}[htbp!]

\scalebox{1.2}{
\def\arraystretch{1.2}
\begin{tabular}{|c|c|c|c|c|c|c|c|c|c|c|c|}
\hline
$x$ & $0$ & $1$ & $2$ & $3$ & $4$ & $5$ & $6$ & $7$ & $8$ & $9$ & $10$ \\ \hline
$\mathcal{O}_{A}(x)$ & $\mathcal{P}$ & $\mathcal{N}$ & $\mathcal{P}$ & $\mathcal{N}$ & $\mathcal{P}$ & $\mathcal{N}$ & $\mathcal{P}$ & $\mathcal{N}$ & $\mathcal{P}$ & $\mathcal{N}$ & $\mathcal{P}$   \\ \hline
$\mathcal{O}_{B}(x)$ & $\mathcal{P}$ & $\mathcal{N}$ & $\mathcal{P}$ & $\mathcal{N}$ & $\mathcal{P}$ & $\mathcal{N}$ & $\mathcal{P}$ & $\mathcal{N}$ & $\mathcal{P}$ & $\mathcal{N}$ & $\mathcal{P}$   \\ \hline
$\mathcal{O}_{ C}(x)$ & $\mathcal{P}$ & $\mathcal{P}$ & $\mathcal{N}$ & $\mathcal{N}$ & $\mathcal{N}$ & $\mathcal{N}$ & $\mathcal{N}$ & $\mathcal{N}$ & $\mathcal{N}$ & $\mathcal{N}$ & $\mathcal{N}$   \\ 
\hline
$\mathcal{O}_{A \odot B}(x)$ & $\mathcal{P}$ & $\mathcal{N}$ & $\mathcal{P}$ & $\mathcal{N}$ & $\mathcal{P}$ & $\mathcal{N}$ & $\mathcal{P}$ & $\mathcal{N}$ & $\mathcal{P}$ & $\mathcal{N}$ & $\mathcal{P}$   \\ \hline

$\mathcal{O}_{A \odot C}(x)$ & $\mathcal{P}$ & $\mathcal{P}$ & $\mathcal{N}$ & $\mathcal{P}$ & $\mathcal{N}$ & $\mathcal{P}$ & $\mathcal{N}$ & $\mathcal{P}$ & $\mathcal{N}$ & $\mathcal{P}$ & $\mathcal{N}$   \\ 
\hline

$\mathcal{O}_{B \odot C}(x)$ & $\mathcal{P}$ & $\mathcal{P}$ & $\mathcal{N}$ & $\mathcal{N}$ & $\mathcal{N}$ & $\mathcal{N}$ & $\mathcal{N}$ & $\mathcal{N}$ & $\mathcal{N}$ & $\mathcal{N}$ & $\mathcal{N}$  \\
\hline

\end{tabular}
}

\vspace{3mm}
\caption{The rows illustrate the first few outcomes of the rulesets $A$, $B$,  $C$,  and  enforce operators   $A \odot B$, $A \odot C$, $B \odot C$,  with $A$, $B$ and $C$ as in the proof of Proposition~\ref{prop:rulesetdomination}. 
}
\label{tab:rulesetA,B,Coutcome}
\end{table}

  From Table~\ref{tab:rulesetA,B,Coutcome}, we have $\mathcal{O}_A(x)=\mathcal{O}_B(x)$, hence $A \simeq B$. In particular $B\vdash A$.
 
 Also, if $\mathcal{O}_C(x)=\mathcal{N}$, we may assume $x \geq 2$, and under ruleset $B$, a player can move to $(x-(2j+1)) \in f_B(x)$, with $j=[\frac{x-1}{2}]$, which is a $\mathcal{P}$-position of ruleset $C$. Hence $C\vdash B$, and so $C\vdash B\vdash A$. 
 However, $P_1(C,A)$ fails to hold. 
 Namely, $\mathcal{O}_C(3) = \mathcal{N}$ but $f_A(3)=\{2\}$ and $\mathcal{O}_C(2) = \mathcal{N}$. Therefore, by Theorem~\ref{thm:dominated_ruleset}, 
 $C\not\,\vdash A$, and transitivity fails. 
\end{proof}
Similarly, one can see that in fact $A \parallel C$, which leads us to the next section.

\subsection{Order Properties}
Domination is not an order, but the following result suggests that there may be some order-like properties for domination.\\

\begin{theorem}
\label{thm_sym}
Consider jointly short rulesets $A$, $B$ and $C$ such that $A\vdash B$, $B\vdash C$ and $C\vdash A$. 
Then the rulesets $A$, $B$ and $C$ are similar.
\end{theorem}
\begin{proof}
Let $E = A \circledcirc B \circledcirc C$, which is short by assumption.

Let $N_A, N_B, N_C \subset \X$ be the sets of $\mathcal{N}$-positions for the ruleset $A$, $B$, and $C$ respectively. Similarly, we let $P_A, P_B, P_C \subset \X$ be the set of $\mathcal{P}$-positions for $A$, $B$, and $C$.

Let us take $M = (N_A \cup N_B \cup N_C) \setminus (N_A \cap N_B \cap N_C)$.
Notice that if we can show that $M$ is
empty, then it implies $N_A = N_B = N_C$, and hence $A$, $B$, and $C$ are similar. Let us
assume that $M$ is non-empty, and $\x \in M$ is an element with a minimal value of $T_E(\x)$. By symmetry, we
may assume that $\x \in N_A$. As $A \vdash B$, by  Definition~\ref{def:property1}, there
must be some option $\y \in f_B(\x) \cap P_A$.
As $\x$ is a minimal in $M, \y \not \in M$, and together
with $\y \not \in N_A$, we can conclude $\y \in P_B$. Now because $\x$ has a $B$-option $\y \in P_B$, we
know that $\x \in N_B$. Then by a similar argument, we have $\x \in N_C$, which contradicts our assumption $\x \in M$. Therefore $M$ must be empty.
\end{proof}

Theorem~\ref{thm_sym}  suggests that domination is very close to an order relation.\\

\begin{definition}[Strong Domination]
\label{def:strongdomi}  
Consider jointly short rulesets $A$ and $B$. Then $A$ strongly dominates $B$, if for any jointly short ruleset $C$, $A \odot B \odot C \simeq A \odot C$.\\

\end{definition}

\begin{proposition}
\label{prop:strongdomination}

Consider jointly short rulesets. 
Strong domination satisfies the following properties.
\begin{enumerate}
\item [($\alpha$)] If ruleset  $A$ strongly dominates ruleset  $B$, then   $A \vdash B$.

\item [($\beta$)] If ruleset $A$ strongly dominates ruleset  $B$ and ruleset $B$ strongly dominates ruleset $C$, then ruleset $A$ strongly dominates ruleset $C$.
\item [($\gamma$)] If ruleset $A$ strongly dominates ruleset $B$ and ruleset $B$ strongly dominates ruleset $C$, then $A \vdash C $.
\item [($\delta$)] When $f_A(\x) \subset f_B(\x)$ for all $\x \in \X$, then ruleset $A$ strongly dominates ruleset $B$. 

\end{enumerate}
\end{proposition}
\begin{proof} 
We prove each property.
\begin{enumerate}
\item[($\alpha$)] Set $C = A$ in the Definition~\ref{def:strongdomi}.

\item[($\beta$)]
For any ruleset $D$,

\begin{align*}
    (A \odot C) \odot D &\simeq  ((A \odot B) \odot C) \odot D ~&~ (A \text{ strongly dominates } B )\\ 
 &\simeq A \odot ((B \odot C) \odot D)~ &(\text{associativity})  \\
&\simeq A \odot (B \odot D)~ &(B \text{ strongly dominates } C)  \\
&\simeq (A \odot B) \odot D &~(\text{associativity})  \\
&\simeq A \odot D &~\hbox{ ($A$ strongly dominates $B$})
\nonumber
 \end{align*}
\item [($\gamma$)] 

This is immediate from ($\alpha$) and ($\beta$).

\item [($\delta$)] As $f_A(\x) \subset f_B(\x)$ for any $\x$, the opponent has no merit for enforcing the ruleset $B$ over $A$.
\end{enumerate}
\end{proof}
\section{A distributive-lattice-like structure}
\label{sec:distributive}
Let us establish some more properties of our operators.\\
Note that the selective operator is not a disjunctive sum operator. In a disjunctive sum of games, the current player chooses one of the components, but in $A \circledcirc B$, the component is single and the player chooses one of the rulesets. In fact, this is the same as playing the combined rulesets $A$ and $B$. However, by regarding it as an operator, we can find a distributive-lattice-like structure on rulesets. For distributive lattice, it is required that, for any  rulesets $A, B$ on $\X$, if $A \simeq B$ then for any ruleset $C$ on $\X$, $A \odot C \simeq B \odot C$. However, due to Proposition \ref{prop:rulesetdomination}, we have a counterexample. Thus, we can only say that we have a {\it distributive-lattice-like}  construction.

Let $E$ be a ruleset on a set $\X$ such that for any non-terminal  $\x \in \X$, $f_E(\x)$ includes the terminal position. Clearly, $\mathcal{O}_E(\x) = \mathcal{P}$ if and only if $\x$ is a terminal position. Then, for any ruleset $A$ on $\X$, $P_1(A, E)$ holds.

Therefore, $$A \simeq A \odot E \simeq E \odot A$$ and $E$ is an identity element of $\odot$.
Furthermore, for any ruleset $A$, $$E \simeq A \circledcirc E \simeq E \circledcirc A$$ because the current player can win from any position except for the terminal position by selecting $E$. Thus, $E$ is an absorbing element of $\circledcirc$.

Let $O$ be a ruleset on a set $\X$ such that for any $\x \in \X$, $f_O(\x)$ is the empty set.
Clearly, $\mathcal{O}_O(\x) = \mathcal{P}$ for any $\x$.
Then, for any ruleset $A$ on $\X$, $P_1(O, A)$ holds.
Therefore, $$O \simeq A \odot O \simeq O \odot A$$ and $O$ is an absorbing element of $\odot$.
Furthermore, for any ruleset $A$, $$A \simeq A \circledcirc O \simeq O \circledcirc A$$ because the current player can only select $A$ for the move.

Thus, $O$ is an identity element of $\circledcirc$.

Obviously, the operators $\odot$ and $\circledcirc$ satisfy the commutative law and the associative law. 

Further, we also have the following theorem.\\

\begin{theorem}[Absorption-Distributive]
\label{thm:Absorption-Distributive}
Suppose $A$, $B$ and $C$ are the rulesets on a set $\X$. The two operators $\odot$ and $\circledcirc$ satisfy the absorption law: 

\begin{eqnarray}
(A \circledcirc B) \odot A &\simeq& A; \nonumber \\
    (A \odot B) \circledcirc A &\simeq& A;\nonumber
\end{eqnarray}
and the distributive law: 
\begin{eqnarray}
    (A \circledcirc B)\odot C &\simeq& (A \odot C) \circledcirc (B \odot C); \nonumber \\
    (A \odot B) \circledcirc C &\simeq& (A \circledcirc C) \odot (B \circledcirc C) \nonumber.
\end{eqnarray}

\end{theorem}

\begin{proof}
First, we show the absorption law. In $(A \circledcirc B) \odot A, $ the opponent has no benefit to choose $(A \circledcirc B)$. Thus, $(A \circledcirc B) \odot A \simeq A$. 

In $(A \odot B) \circledcirc A,$  the current player has no benefit to choose $(A \odot B)$. Thus, $(A \odot B) \circledcirc A \simeq A$.

Next, we demonstrate the distributive law. 

Consider first the case that  $\x$ is an $\mathcal{N}$-position under $(A \circledcirc B)\odot C$. This means that the current player has a winning move in one of $A$ and $B$ and in $C$. Without loss of generality, we may assume that the player has winning moves in $A$ and in $C$. Then, playing from $\x$, under $(A \odot C) \circledcirc (B \odot C),$ the current player can select $A \odot C$ and has a winning strategy whichever ruleset the opponent enforces. 

Next, consider the case that $\x$ is a $\mathcal{P}$-position under $(A \circledcirc B)\odot C$. Then the opponent can win by enforcing $(A \circledcirc B)$ or $C$. Assume that the opponent can win by enforcing $(A \circledcirc B)$. Then, playing in $\x$ under $(A \odot C) \circledcirc (B \odot C)$, whichever the current player chooses $(A \odot C)$ or $(B \odot C)$, the opponent can win by enforcing $A$ or $B$. If the opponent can win by enforcing $C$, then in $\x$ under $(A \odot C) \circledcirc (B \odot C)$ the opponent can also win because whichever the current player chooses, the opponent can  enforce 
$C$ at last.

Secondly, if $\x$ is an $\mathcal{N}$-position under $(A \odot B) \circledcirc C$, the current player has a winning move in both of $A$ and $B$ or in $C$.
Assume that the player has winning strategies in both $A$ and $B$. Then, in $\x$ under $(A \circledcirc C) \odot (B \circledcirc C), $ whichever the opponent enforces,  
the current player can choose a winning move in $A$ or in $B$.  If the current player has a winning move in $C$, then in $\x$ under $(A \circledcirc C) \odot (B \circledcirc C), $ whichever the opponent enforces, 
the current player can choose a winning move in $C$. 

If $\x$ is a $\mathcal{P}$-position under $(A \odot B) \circledcirc C,$ then the opponent can win in $(A \odot B)$ and in $C$. Without loss of generality, we may assume that the opponent can win by enforcing $A$. Then, in $\x$ under $(A \circledcirc C) \odot (B \circledcirc C)$, the opponent can win by enforcing $(A \circledcirc C)$. Therefore, the pair of operators $(\odot, \circledcirc)$ 
satisfies the distributive law.
 \end{proof}

\section{Nimbers under enforced rulesets}
\label{sec:nimvalueforcomply}

Let $\x=(x_1, x_2, \ldots, x_k)$ be a {\sc Nim} position with heap sizes $x_1, x_2, \ldots, x_k$. The 
nim-sum of $\x$ is given by $x_1 \oplus x_2 \oplus \cdots \oplus x_k$, where the $\oplus$ operator is:  write the numbers $x_1, x_2, \ldots, x_k$ in binary and then add them without carrying. If the nim-sum is zero, we call $\x$ a zero position. 
\\

\begin{theorem}[Bouton \cite{B02}]
\label{thm:bouton}
Let $\x$ be a {\sc Nim} position.
\begin{enumerate}
    \item If $\x$ is a zero position, then every move from $\x$ leads to a nonzero position.
    \item If $\x$ is a nonzero position, then there exists a move from $\x$ to a zero position.
    \\
\end{enumerate}
\end{theorem}
Therefore, if $\x$ is a zero position, the  previous 
player can guarantee a win i.e., $\x$ is a $\mathcal{P}$-position. If $\x$ is a nonzero position, then the current 
player can guarantee a win i.e., $\x$ is an $\mathcal{N}$-position.

Sprague and Grundy extended Bouton’s theorem for general impartial games in normal play. The key is the minimal excludant (mex) function.\\
\begin{definition}[Minimal Excludant]\
\label{def:mex}
Let $S$ denote a finite set of non-negative integers. Then the minimal excludant
$\mathrm{mex}(S)$ is the least non-negative integer not in $S$. \\
\end{definition}

The Sprague-Grundy theory recursively assigns a non-negative integer $\mathcal{G}(\x)$, known as the nim-value of $\x$, to each short normal play impartial game $\x$.\\

\begin{definition}[Nim-value]
\label{def:nimvalue}
For any game position $\x$, the nim-value function $\mathcal{G}(\x)$ is
$$
\mathcal{G}(\x)= \mathrm{mex}(\{\mathcal{G}(\x') \mid \x'\in f(\x)\}).
$$
\end{definition}

We sometimes abbreviate ``nim-value'' by {\em nimber}.\\

\begin{theorem}[Sprague-Grundy \cite{SP35,GR39}]
\label{thm:Sprague-Grundy}
     For any impartial normal-play game position $\x$, $\x$ is a $\mathcal{P}$-position if and only if $\mathcal{G}(\x) = 0$. In general, $\x$ is equivalent to a nim heap of size $\mathcal{G}(\x)$, i.e. the nimber $\mathcal{G}(\x)$.\\
\end{theorem}

\begin{definition}[Disjunctive Sum]
\label{def:sumgame}
For any game positions $\x$ and $\y$, the disjunctive sum $\x + \y$ is the game whose options are $\x+f(\y)=\{\x+\y'\mid \y'\in f(\y)\}$ and $f(\x)+\y=\{\x'+\y\mid \x'\in f(\x)\}$.
\\
\end{definition}
\begin{theorem}[Sprague-Grundy \cite{SP35,GR39}]
\label{thm:Sprague-Grundygeneral}
    For any game positions $\x$ and $\y$,
$$
\mathcal{G}(\x+\y)= \mathcal{G}(\x) \oplus  \mathcal{G}(\y).
    $$

\end{theorem}

We can also consider the disjunctive sum of positions under rulesets combined by the enforce operator. We define recursively the enforce nim-value  $\mathcal{G^E}(\x)$ of such positions as follows:\\

\begin{definition}[Enforce Operator Nimbers]
\label{Def:enforce_nimval}

Let $A$ and $B$ be jointly short rulesets on $\X$, and consider the game $A\odot B$.  
  Then, the enforce nimber of $\x \in \X$ is $\mathcal{G^E}(\x)=\mathcal{G^E}_{A\odot B}(\x)$ where $$\mathcal{G^E}(\x) ={\rm mex}(\{\mathcal{G^E}(\x')\, |\, \x'\in f_A(\x)\}\cap \{\mathcal{G^E}(\x')\, |\, \x'\in f_B(\x)\}).\footnote{This definition is similar to the way of \cite{HR1}, but in that manuscript they do not consider the concept of operator for rulesets.}$$
\end{definition}

Let us justify the soundness of  Definition~\ref{Def:enforce_nimval} before we state the main result of this section. As $A$ and $B$ are  jointly short, for any $\x\in \X$, there exists some number $T_{A \circledcirc B}(\x)$ such that every play sequence of $A\circledcirc B$ starting from $\x$ terminates at most $T_{A \circledcirc B}(\x)$ moves.\footnote{Here, we will use induction on the number of the length of play.
 Then, we use $T_{A \circledcirc B}(\x)$ (this $T$ is defined in Definition \ref{def:shortgame}) for the upper bound of the length of play on $A \odot B$, because we cannot define $T_{A \odot B}$.}  To find  $\mathcal{G^E}(\x)=\mathcal{G^E}_{\!A\odot B}(\x)$, we proceed by induction on $T_{A \circledcirc B}(\x)$.  If $T_{A \circledcirc B}(\x)=0$, then $\mathcal{G^E}(\x)=0$.  Assume that for any ${\y}$ with $T_{A \circledcirc B}({\y})<T_{A \circledcirc B}(\x)$, we have already determined $\mathcal{G^E}({\y})$.  Then for each $\x'\in f_A(\x)\cup f_B(\x)$, we have $T_{A \circledcirc B}(\x')<T_{A \circledcirc B}(\x)$, and we may assume that $\mathcal{G}(\x')$ is already determined.\\

\begin{theorem} Let $A$ and $B$ be   jointly short rulesets on $\X$. 
\label{thm:nimvalue}
\begin{enumerate}

\item[(1)] If $f_A(\x)=\emptyset$ or $f_B(\x)=\emptyset$, then $\mathcal{G^E}(\x)=0$.
\item[(2)] If $\mathcal{G^E}(\x)=n$, then the previous player can enforce a ruleset 
such that, for any option $\x'$, $\mathcal{G^E}(\x')\ne n$. 

\item[(3)] Consider integers $0\le m<n$. If $\mathcal{G^E}(\x)= n$, then,  whichever ruleset $A$ or $B$ the previous player enforces, there is an option $\x'$ such that $\mathcal{G^E}(\x') = m$.
\item[(4)] $\x\in \X$ is a $\mathcal{P}$-position of $A\odot B$ if and only if $\mathcal{G^E}(\x)=0$.
\item[(5)] 
Let $R_1, \ldots, R_k$ be impartial rulesets, some of them being impartial short rulesets, say $R_1, \ldots, R_j$ and some of them being jointly short enforce rulesets, say $R_{j+1},\ldots, R_k$.  The corresponding game boards are $\X_1,\ldots, \X_k$. 
The  disjunctive sum $\x_1+\cdots+\x_k$ is a $\mathcal{P}$-position if and only if the nim-sum $$\mathcal{G}(\x_1)\oplus \cdots \oplus \mathcal{G}(\x_j)\oplus \mathcal{G^E}(\x_{j+1})\oplus \cdots \oplus \mathcal{G^E}(\x_k)=0.$$  In other words, the enforced nimbers can be used in exactly the same way as the classical nimbers for determining the outcomes of disjunctive sums. In particular, when $A\odot B$ is an enforce operator of two jointly short rulesets $A$ and $B$ on a set $\X$, and $*n$ is a {\sc Nim} heap of size $n$, 
then $\x + *n$ is a $\mathcal{P}$-position if and only if $\mathcal {G^E}(\x)=n$.
\end{enumerate}
\end{theorem}

\begin{proof}

We demonstrate that  $\mathcal{G^E}(\x)$ as in Definition~\ref{Def:enforce_nimval} satisfies all the properties (1)$\sim$(5).

(1)  Suppose without loss of generality that  $f_A(\x)=\emptyset$; then $\{\mathcal{G^E}(\x')\, |\, \x'\in f_A(\x)\}=\emptyset$, and  the proof follows immediately from Definition~\ref{Def:enforce_nimval}. 

(2) Assume that $\mathcal{G^E}(\x)=n\in \mathbb{Z}_{\ge 0}$. Then, by Definition~\ref{Def:enforce_nimval}, either, for all $\x'\in f_A(\x)$, $\mathcal{G^E}(\x')\ne n$ , or, for all $\x'\in f_B(\x)$,  $\mathcal{G^E}(\x')\ne n$.  In the first case enforce the ruleset $A$ and otherwise $B$. 

(3) Again we assume that $\mathcal{G^E}(\x)=n\in \mathbb{Z}_{\ge 0}$. Take a non-negative integer $m~(<n)$. Then by Definition~\ref{Def:enforce_nimval}, $m\in \{\mathcal{G^E}(\x')\, |\, \x'\in f_A(\x)\}\cap \{\mathcal{G^E}(\x')\, |\, \x'\in f_B(\x)\}$, which implies $\mathcal{G^E}(\x')=m$ for some $\x'\in f_A(\x)$ and $\mathcal{G^E}(\x')=m$ for some $\x'\in f_B(\x)$.  Therefore, whichever ruleset the opponent enforces, the current player can find an option $\x'$ such that $\mathcal{G^E}(\x')=m$.

(4) Set $n=0$ in (2) and set $m=0$ in (3). Observe that in (2), the enforcing previous player has the usual advantage in a $\mathcal{P}$-position, while in (3) the non-enforcing current player has the standard $\mathcal N$-position advantage.  
 (5) Items (1)-(4) establish that the enforced nimbers $\mathcal{G^E}$ satisfy all the standard properties of classical nimbers. Consequently, we can add them using the nim-sum in the same manner as classical nimbers to determine the outcomes of disjunctive sums. 
\end{proof}
\section{Game practice}
\label{sec:examples}
We will now consider various examples and compute the nimbers under the enforce operator, solve the initial problem, and more.\\ 

\textit{}\begin{example}
\label{ex:nim_value}
 Consider the rulesets {\sc Yama Nim}, $\large\rook_{\bf Y}$ and {\sc Wythoff Nim}, $\large\queen$. Figure~\ref{fig:YNWN} displays the nimbers for an initial game board with these ruelsets and the enforce combination  $ \large\rook_{\bf Y}\odot \large\queen$, respectively.  We observe that the nimbers of $\large\rook_{\bf Y}$ coincide with those of $\large\rook_{\bf Y}\odot\large\queen$. Thus playing a disjunctive sum $X+Y$, where $X\in\large\rook_{\bf Y}\odot\large\queen  $ and $Y\in \large\queen$, is the same as if $X\in\large\rook_{\bf Y}$ and $Y\in \large\queen$.  
\end{example}
\definecolor{teal}{rgb}{0.0, 0.7, 0.0}
\definecolor{green}{rgb}{0.78, 0.0, 0.94}
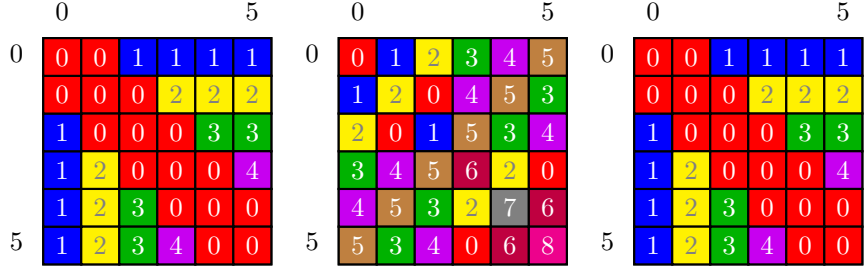
\begin{figure}[htbp!]

\begin{center}
\begin{tikzpicture}[scale = 0.5,
         box/.style={rectangle,draw=black, thick, minimum size=.5 cm},
     ]

\foreach \x in {0,1,...,5}{
     \foreach \y in {0,1,...,5}
         \node[box] at (\x,\y){};
 }

\draw (0, 6.2) node {$0$};
\draw (5, 6.2) node {$5$};

 \draw (-1.2, 5.1) node {$0$};
 \draw (-1.2, 0.1) node {$5$};

\node[box,fill=blue] at (0,0){};
\node[box,fill=blue] at (0,1){};
\node[box,fill=blue] at (0,2){};
\node[box,fill=blue] at (0,3){};
\node[box,fill=red] at (0,4){};
\node[box,fill=red] at (0,5){};
\node[box,fill=red] at (1,5){};
\node[box,fill=blue] at (2,5){};
\node[box,fill=blue] at (3,5){};
\node[box,fill=blue] at (4,5){};
\node[box,fill=blue] at (5,5){};
\node[box,fill=red] at (1,4){};
\node[box,fill=red] at (2,4){};
\node[box,fill=yellow] at (3,4){};
\node[box,fill=yellow] at (4,4){};
\node[box,fill=yellow] at (5,4){};
\node[box,fill=red] at (1,3){};
\node[box,fill=red] at (2,3){};
\node[box,fill=red] at (3,3){};
\node[box,fill=teal] at (4,3){};
\node[box,fill=teal] at (5,3){};
\node[box,fill=yellow] at (1,2){};
\node[box,fill=red] at (2,2){};
\node[box,fill=red] at (3,2){};
\node[box,fill=red] at (4,2){};
\node[box,fill=green] at (5,2){};
\node[box,fill=yellow] at (1,1){};
\node[box,fill=teal] at (2,1){};
\node[box,fill=red] at (3,1){};
\node[box,fill=red] at (4,1){};
\node[box,fill=red] at (5,1){};
\node[box,fill=yellow] at (1,0){};
\node[box,fill=teal] at (2,0){};
\node[box,fill=green] at (3,0){};
\node[box,fill=red] at (4,0){};
\node[box,fill=red] at (5,0){};

\draw[white] (0,0) node {1};
\draw[white] (0,1) node {1};
\draw[white] (0,2) node {1};
\draw[white] (0,3) node {1};
\draw[white] (0,4) node {0};
\draw[white] (0,5) node {0};
\draw[white] (1,5) node {0};
\draw[white] (2,5) node {1};
\draw[white] (3,5) node {1};
\draw[white] (4,5) node {1};
\draw[white] (5,5) node {1};
\draw[white] (1,4) node {0};
\draw[white] (2,4) node {0};
\draw[gray] (3,4) node {2};
\draw[gray] (4,4) node {2};
\draw[gray] (5,4) node {2};
\draw[white] (1,3) node {0};
\draw[white] (2,3) node {0};
\draw[white] (3,3) node {0};
\draw[white] (4,3) node {3};
\draw[white] (5,3) node {3};
\draw[gray] (1,2) node {2};
\draw[white] (2,2) node {0};
\draw[white] (3,2) node {0};
\draw[white] (4,2) node {0};
\draw[white] (5,2) node {4};
\draw[gray] (1,1) node {2};
\draw[white] (2,1) node {3};
\draw[white] (3,1) node {0};
\draw[white] (4,1) node {0};
\draw[white] (5,1) node {0};
\draw[gray] (1,0) node {2};
\draw[white] (2,0) node {3};
\draw[white] (3,0) node {4};
\draw[white] (4,0) node {0};
\draw[white] (5,0) node {0};
\end{tikzpicture}\hspace{2 mm}
\begin{tikzpicture}[scale = 0.5,
         box/.style={rectangle,draw=black, thick, minimum size=.5 cm},
     ]
\foreach \x in {0,1,...,5}{
     \foreach \y in {0,1,...,5}
         \node[box] at (\x,\y){};
 }
\draw (0, 6.2) node {$0$};
\draw (5, 6.2) node {$5$};
\draw (-1.2, 5.1) node {$0$};
\draw (-1.2, 0.1) node {$5$};
\node[box,fill=brown] at (0,0){};
\node[box,fill=green] at (0,1){};
\node[box,fill=teal] at (0,2){};
\node[box,fill=yellow] at (0,3){};
\node[box,fill=blue] at (0,4){};
\node[box,fill=red] at (0,5){};
\node[box,fill=blue] at (1,5){};
\node[box,fill=yellow] at (2,5){};
\node[box,fill=teal] at (3,5){};
\node[box,fill=green] at (4,5){};
\node[box,fill=brown] at (5,5){};
\node[box,fill=yellow] at (1,4){};
\node[box,fill=red] at (2,4){};
\node[box,fill= green] at (3,4){};
\node[box,fill=brown] at (4,4){};
\node[box,fill=teal] at (5,4){};
\node[box,fill=red] at (1,3){};
\node[box,fill=blue] at (2,3){};
\node[box,fill= brown] at (3,3){};
\node[box,fill=teal] at (4,3){};
\node[box,fill= green] at (5,3){};
\node[box,fill=green] at (1,2){};
\node[box,fill=brown] at (2,2){};
\node[box,fill= purple] at (3,2){};
\node[box,fill=yellow] at (4,2){};
\node[box,fill= red] at (5,2){};
\node[box,fill=brown] at (1,1){};
\node[box,fill=teal] at (2,1){};
\node[box,fill=yellow] at (3,1){};
\node[box,fill=gray] at (4,1){};
\node[box,fill= purple] at (5,1){};
\node[box,fill=teal] at (1,0){};
\node[box,fill=green] at (2,0){};
\node[box,fill=red] at (3,0){};
\node[box,fill=purple] at (4,0){};
\node[box,fill= magenta] at (5,0){};
  \draw[white] (0,0) node {5};
  \draw[white] (0,1) node {4};
  \draw[white] (0,2) node {3};
  \draw[gray] (0,3) node {2};
  \draw[white] (0,4) node {1};
  \draw[white] (0,5) node {0};
 \draw[white] (1,5) node {1};
 \draw[gray] (2,5) node {2};
 \draw[white] (3,5) node {3};
 \draw[white] (4,5) node {4};
 \draw[white] (5,5) node {5};
\draw[gray] (1,4) node {2};
\draw[white] (2,4) node {0};
\draw[white] (3,4) node {4};
\draw[white] (4,4) node {5};
\draw[white] (5,4) node {3};
\draw[white] (1,3) node {0};
\draw[white] (2,3) node {1};
\draw[white] (3,3) node {5};
\draw[white] (4,3) node {3};
\draw[white] (5,3) node {4};
\draw[white] (1,2) node {4};
\draw[white] (2,2) node {5};
\draw[white] (3,2) node {6};
\draw[gray] (4,2) node {2};
\draw[white] (5,2) node {0};
\draw[white] (1,1) node {5};
\draw[white] (2,1) node {3};
\draw[gray] (3,1) node {2};
\draw[white] (4,1) node {7};
\draw[white] (5,1) node {6};
\draw[white] (1,0) node {3};
\draw[white] (2,0) node {4};
\draw[white] (3,0) node {0};
\draw[white] (4,0) node {6};
\draw[white] (5,0) node {8};
\end{tikzpicture}\hspace{2 mm}
\begin{tikzpicture}[scale = 0.5,
         box/.style={rectangle,draw=black, thick, minimum size=.5 cm},
     ]
\foreach \x in {0,1,...,5}{
     \foreach \y in {0,1,...,5}
         \node[box] at (\x,\y){};
 }
\draw (0, 6.2) node {$0$};
\draw (5, 6.2) node {$5$};
\draw (-1.2, 5.1) node {$0$};
\draw (-1.2, 0.1) node {$5$};

\node[box,fill=blue] at (0,0){};
\node[box,fill=blue] at (0,1){};
\node[box,fill=blue] at (0,2){};
\node[box,fill=blue] at (0,3){};
\node[box,fill=red] at (0,4){};
\node[box,fill=red] at (0,5){};
\node[box,fill=red] at (1,5){};
\node[box,fill=blue] at (2,5){};
\node[box,fill=blue] at (3,5){};
\node[box,fill=blue] at (4,5){};
\node[box,fill=blue] at (5,5){};
\node[box,fill=red] at (1,4){};
\node[box,fill=red] at (2,4){};
\node[box,fill=yellow] at (3,4){};
\node[box,fill=yellow] at (4,4){};
\node[box,fill=yellow] at (5,4){};
\node[box,fill=red] at (1,3){};
\node[box,fill=red] at (2,3){};
\node[box,fill=red] at (3,3){};
\node[box,fill=teal] at (4,3){};
\node[box,fill=teal] at (5,3){};
\node[box,fill=yellow] at (1,2){};
\node[box,fill=red] at (2,2){};
\node[box,fill=red] at (3,2){};
\node[box,fill=red] at (4,2){};
\node[box,fill=green] at (5,2){};
\node[box,fill=yellow] at (1,1){};
\node[box,fill=teal] at (2,1){};
\node[box,fill=red] at (3,1){};
\node[box,fill=red] at (4,1){};
\node[box,fill=red] at (5,1){};
\node[box,fill=yellow] at (1,0){};
\node[box,fill=teal] at (2,0){};
\node[box,fill=green] at (3,0){};
\node[box,fill=red] at (4,0){};
\node[box,fill=red] at (5,0){};

\draw[white] (0,0) node {1};
\draw[white] (0,1) node {1};
\draw[white] (0,2) node {1};
\draw[white] (0,3) node {1};
\draw[white] (0,4) node {0};
\draw[white] (0,5) node {0};
\draw[white] (1,5) node {0};
\draw[white] (2,5) node {1};
\draw[white] (3,5) node {1};
\draw[white] (4,5) node {1};
\draw[white] (5,5) node {1};
\draw[white] (1,4) node {0};
\draw[white] (2,4) node {0};
\draw[gray] (3,4) node {2};
\draw[gray] (4,4) node {2};
\draw[gray] (5,4) node {2};
\draw[white] (1,3) node {0};
\draw[white] (2,3) node {0};
\draw[white] (3,3) node {0};
\draw[white] (4,3) node {3};
\draw[white] (5,3) node {3};
\draw[gray] (1,2) node {2};
\draw[white] (2,2) node {0};
\draw[white] (3,2) node {0};
\draw[white] (4,2) node {0};
\draw[white] (5,2) node {4};
\draw[gray] (1,1) node {2};
\draw[white] (2,1) node {3};
\draw[white] (3,1) node {0};
\draw[white] (4,1) node {0};
\draw[white] (5,1) node {0};
\draw[gray] (1,0) node {2};
\draw[white] (2,0) node {3};
\draw[white] (3,0) node {4};
\draw[white] (4,0) node {0};
\draw[white] (5,0) node {0};
\end{tikzpicture}
\end{center}
\vspace{3mm}
\caption{The initial nimbers (0: red, 1: blue, 2: yellow, 3: green, 4: purple, 5: brown, 6: merlot, 7: gray, 8: magenta) of {\sc Yama Nim}, $\large\rook_{\bf Y}$ (left) {\sc Wythoff Nim} $\large \queen$ (middle) and 
$\large\rook_{\bf Y} \odot \large\queen$ (right). The nimbers in the rightmost picture are computed as follows: if the value is $v$, then, for all $v'<v$, each ruleset has an option of value $v'$, and at most one ruleset has an option of value $v$. 
}\label{fig:YNWN}
\end{figure}

\begin{example} 
\label{ex:odawara_nim_value}
Let $\X=\mathbb{Z}_{\geq 0} \times \mathbb{Z}_{\geq 0}$ and consider the  rulesets $C$ and $D$, defined by 
$f_C((x, y))=\left\{(x-i, y), (x,y-i) \, |\, i \in \{1, 2\},~ x,y \geq i\right\}$, and\\ 

 $f_D((x, y))=\left\{
  \begin{array}{ll}
    \{(x-i, y-j) \, |\, i, j \in \{1, 2\},~ x \geq i,~ y \geq j \}, & \text{if } x,y>0\\
    \{(x-i, y) \, |\, i \in \{1, 2\},~ x \geq i \},& \text{if } x>0,y=0\\
    \{(x, y-i) \, |\, i \in \{1, 2\},~ y \geq i \},& \text{if } x=0,y>0\\
    \emptyset, & \text{if } x,y=0.
  \end{array}
  \right.$\\\\
 Here ruleset $C$ is the usual disjunctive sum of two subtraction games, both with the set of removable numbers $S=\{1,2 \}$. 
And ruleset $D$ is the Continued Conjunctive Sum \cite{C76, S13} of two copies of $S$ \cite{OD24}. 
In Figures~\ref{fig:CandD} and \ref{fig:CenforceD}, we have computed the nimbers for the rulesets $C$, $D$, and the enforce combination $C \odot D$, respectively. Observe that the nimbers of the two individual games both differ from those under the enforce combination $C \odot D$. Moreover, the $\mathcal{P}$-positions are distinct. Therefore, rulesets $C$ and $D$ are confused.\\
\end{example}

\begin{figure}[htbp!]
    \centering
\begin{tikzpicture}[scale = 0.5,
         box/.style={rectangle,draw=black, thick, minimum size=.5 cm},
     ]
 \foreach \x in {0,1,...,10}{
     \foreach \y in {0,1,...,10}
         \node[box] at (\x,\y){};
 }
\draw (0, 11.2) node {$0$};
\draw (10, 11.2) node {$10$};
\draw (-1.2, 10.1) node {$0$};
\draw (-1.3, 0.1) node {$10$};
\node[box,fill=red] at (0,10){};
\node[box,fill=blue] at (1,10){};
\node[box,fill=yellow] at (2,10){};
\node[box,fill=red] at (3,10){};
\node[box,fill=blue] at (4,10){};
\node[box,fill=yellow] at (5,10){};
\node[box,fill=red] at (6,10){};
\node[box,fill=blue] at (7,10){};
\node[box,fill=yellow] at (8,10){};
\node[box,fill=red] at (9,10){};
\node[box,fill=blue] at (10,10){};
\node[box,fill=blue] at (0,9){};
\node[box,fill=red] at (1,9){};
\node[box,fill=teal] at (2,9){};
\node[box,fill=blue] at (3,9){};
\node[box,fill=red] at (4,9){};
\node[box,fill=teal] at (5,9){};
\node[box,fill=blue] at (6,9){};
\node[box,fill=red] at (7,9){};
\node[box,fill=teal] at (8,9){};
\node[box,fill=blue] at (9,9){};
\node[box,fill=red] at (10,9){};
\node[box,fill=yellow] at (0,8){};
\node[box,fill=teal] at (1,8){};
\node[box,fill=red] at (2,8){};
\node[box,fill=yellow] at (3,8){};
\node[box,fill=teal] at (4,8){};
\node[box,fill=red] at (5,8){};
\node[box,fill=yellow] at (6,8){};
\node[box,fill=teal] at (7,8){};
\node[box,fill=red] at (8,8){};
\node[box,fill=yellow] at (9,8){};
\node[box,fill=teal] at (10,8){};
\node[box,fill=red] at (0,7){};
\node[box,fill=blue] at (1,7){};
\node[box,fill=yellow] at (2,7){};
\node[box,fill=red] at (3,7){};
\node[box,fill=blue] at (4,7){};
\node[box,fill=yellow] at (5,7){};
\node[box,fill=red] at (6,7){};
\node[box,fill=blue] at (7,7){};
\node[box,fill=yellow] at (8,7){};
\node[box,fill=red] at (9,7){};
\node[box,fill=blue] at (10,7){};
\node[box,fill=blue] at (0,6){};
\node[box,fill=red] at (1,6){};
\node[box,fill=teal] at (2,6){};
\node[box,fill=blue] at (3,6){};
\node[box,fill=red] at (4,6){};
\node[box,fill=teal] at (5,6){};
\node[box,fill=blue] at (6,6){};
\node[box,fill=red] at (7,6){};
\node[box,fill=teal] at (8,6){};
\node[box,fill=blue] at (9,6){};
\node[box,fill=red] at (10,6){};
\node[box,fill=yellow] at (0,5){};
\node[box,fill=teal] at (1,5){};
\node[box,fill=red] at (2,5){};
\node[box,fill=yellow] at (3,5){};
\node[box,fill=teal] at (4,5){};
\node[box,fill=red] at (5,5){};
\node[box,fill=yellow] at (6,5){};
\node[box,fill=teal] at (7,5){};
\node[box,fill=red] at (8,5){};
\node[box,fill=yellow] at (9,5){};
\node[box,fill=teal] at (10,5){};
\node[box,fill=red] at (0,4){};
\node[box,fill=blue] at (1,4){};
\node[box,fill=yellow] at (2,4){};
\node[box,fill=red] at (3,4){};
\node[box,fill=blue] at (4,4){};
\node[box,fill=yellow] at (5,4){};
\node[box,fill=red] at (6,4){};
\node[box,fill=blue] at (7,4){};
\node[box,fill=yellow] at (8,4){};
\node[box,fill=red] at (9,4){};
\node[box,fill=blue] at (10,4){};
\node[box,fill=blue] at (0,3){};
\node[box,fill=red] at (1,3){};
\node[box,fill=teal] at (2,3){};
\node[box,fill=blue] at (3,3){};
\node[box,fill=red] at (4,3){};
\node[box,fill=teal] at (5,3){};
\node[box,fill=blue] at (6,3){};
\node[box,fill=red] at (7,3){};
\node[box,fill=teal] at (8,3){};
\node[box,fill=blue] at (9,3){};
\node[box,fill=red] at (10,3){};
\node[box,fill=yellow] at (0,2){};
\node[box,fill=teal] at (1,2){};
\node[box,fill=red] at (2,2){};
\node[box,fill=yellow] at (3,2){};
\node[box,fill=teal] at (4,2){};
\node[box,fill=red] at (5,2){};
\node[box,fill=yellow] at (6,2){};
\node[box,fill=teal] at (7,2){};
\node[box,fill=red] at (8,2){};
\node[box,fill=yellow] at (9,2){};
\node[box,fill=teal] at (10,2){};
\node[box,fill=red] at (0,1){};
\node[box,fill=blue] at (1,1){};
\node[box,fill=yellow] at (2,1){};
\node[box,fill=red] at (3,1){};
\node[box,fill=blue] at (4,1){};
\node[box,fill=yellow] at (5,1){};
\node[box,fill=red] at (6,1){};
\node[box,fill=blue] at (7,1){};
\node[box,fill=yellow] at (8,1){};
\node[box,fill=red] at (9,1){};
\node[box,fill=blue] at (10,1){};
\node[box,fill=blue] at (0,0){};
\node[box,fill=red] at (1,0){};
\node[box,fill=teal] at (2,0){};
\node[box,fill=blue] at (3,0){};
\node[box,fill=red] at (4,0){};
\node[box,fill=teal] at (5,0){};
\node[box,fill=blue] at (6,0){};
\node[box,fill=red] at (7,0){};
\node[box,fill=teal] at (8,0){};
\node[box,fill=blue] at (9,0){};
\node[box,fill=red] at (10,0){};
\draw[white] (0,10) node {0};
\draw[white] (1,10) node {1};
\draw[gray] (2,10) node {2};
\draw[white] (3,10) node {0};
\draw[white] (4,10) node {1};
\draw[gray] (5,10) node {2};
\draw[white] (6,10) node {0};
\draw[white] (7,10) node {1};
\draw[gray] (8,10) node {2};
\draw[white] (9,10) node {0};
\draw[white] (10,10) node {1};
\draw[white] (0,9) node {1};
\draw[white] (1,9) node {0};
\draw[white] (2,9) node {3};
\draw[white] (3,9) node {1};
\draw[white] (4,9) node {0};
\draw[white] (5,9) node {3};
\draw[white] (6,9) node {1};
\draw[white] (7,9) node {0};
\draw[white] (8,9) node {3};
\draw[white] (9,9) node {1};
\draw[white] (10,9) node {0};
\draw[gray] (0,8) node {2};
\draw[white] (1,8) node {3};
\draw[white] (2,8) node {0};
\draw[gray] (3,8) node {2};
\draw[white] (4,8) node {3};
\draw[white] (5,8) node {0};
\draw[gray] (6,8) node {2};
\draw[white] (7,8) node {3};
\draw[white] (8,8) node {0};
\draw[gray] (9,8) node {2};
\draw[white] (10,8) node {3};
\draw[white] (0,7) node {0};
\draw[white] (1,7) node {1};
\draw[gray] (2,7) node {2};
\draw[white] (3,7) node {0};
\draw[white] (4,7) node {1};
\draw[gray] (5,7) node {2};
\draw[white] (6,7) node {0};
\draw[white] (7,7) node {1};
\draw[gray] (8,7) node {2};
\draw[white] (9,7) node {0};
\draw[white] (10,7) node {1};
\draw[white] (0,6) node {1};
\draw[white] (1,6) node {0};
\draw[white] (2,6) node {3};
\draw[white] (3,6) node {1};
\draw[white] (4,6) node {0};
\draw[white] (5,6) node {3};
\draw[white] (6,6) node {1};
\draw[white] (7,6) node {0};
\draw[white] (8,6) node {3};
\draw[white] (9,6) node {1};
\draw[white] (10,6) node {0};
\draw[gray] (0,5) node {2};
\draw[white] (1,5) node {3};
\draw[white] (2,5) node {0};
\draw[gray] (3,5) node {2};
\draw[white] (4,5) node {3};
\draw[white] (5,5) node {0};
\draw[gray] (6,5) node {2};
\draw[white] (7,5) node {3};
\draw[white] (8,5) node {0};
\draw[gray] (9,5) node {2};
\draw[white] (10,5) node {3};
\draw[white] (0,4) node {0};
\draw[white] (1,4) node {1};
\draw[gray] (2,4) node {2};
\draw[white] (3,4) node {0};
\draw[white] (4,4) node {1};
\draw[gray] (5,4) node {2};
\draw[white] (6,4) node {0};
\draw[white] (7,4) node {1};
\draw[gray] (8,4) node {2};
\draw[white] (9,4) node {0};
\draw[white] (10,4) node {1};
\draw[white] (0,3) node {1};
\draw[white] (1,3) node {0};
\draw[white] (2,3) node {3};
\draw[white] (3,3) node {1};
\draw[white] (4,3) node {0};
\draw[white] (5,3) node {3};
\draw[white] (6,3) node {1};
\draw[white] (7,3) node {0};
\draw[white] (8,3) node {3};
\draw[white] (9,3) node {1};
\draw[white] (10,3) node {0};
\draw[gray] (0,2) node {2};
\draw[white] (1,2) node {3};
\draw[white] (2,2) node {0};
\draw[gray] (3,2) node {2};
\draw[white] (4,2) node {3};
\draw[white] (5,2) node {0};
\draw[gray] (6,2) node {2};
\draw[white] (7,2) node {3};
\draw[white] (8,2) node {0};
\draw[gray] (9,2) node {2};
\draw[white] (10,2) node {3};
\draw[white] (0,1) node {0};
\draw[white] (1,1) node {1};
\draw[gray] (2,1) node {2};
\draw[white] (3,1) node {0};
\draw[white] (4,1) node {1};
\draw[gray] (5,1) node {2};
\draw[white] (6,1) node {0};
\draw[white] (7,1) node {1};
\draw[gray] (8,1) node {2};
\draw[white] (9,1) node {0};
\draw[white] (10,1) node {1};
\draw[white] (0,0) node {1};
\draw[white] (1,0) node {0};
\draw[white] (2,0) node {3};
\draw[white] (3,0) node {1};
\draw[white] (4,0) node {0};
\draw[white] (5,0) node {3};
\draw[white] (6,0) node {1};
\draw[white] (7,0) node {0};
\draw[white] (8,0) node {3};
\draw[white] (9,0) node {1};
\draw[white] (10,0) node {0};
\end{tikzpicture}\hspace{2 mm}
\begin{tikzpicture}[scale = 0.5,
         box/.style={rectangle,draw=black, thick, minimum size=.5 cm},
     ]
 \foreach \x in {0,1,...,10}{
     \foreach \y in {0,1,...,10}
         \node[box] at (\x,\y){};
 }
\draw[thick, color=gray!60!yellow] (-0.55,-0.55) -- (-0.55, 10.55);
\draw[thick, color=gray!60!yellow] (-0.55, 10.55)-- (10.55, 10.55);
\draw (0, 11.2) node {$0$};
\draw (10, 11.2) node {$10$};
\draw (-1.2, 10.1) node {$0$};
\draw (-1.3, 0.1) node {$10$};
\node[box,fill=red] at (0,10){};
\node[box,fill=blue] at (1,10){};
\node[box,fill=yellow] at (2,10){};
\node[box,fill=red] at (3,10){};
\node[box,fill=blue] at (4,10){};
\node[box,fill=yellow] at (5,10){};
\node[box,fill=red] at (6,10){};
\node[box,fill=blue] at (7,10){};
\node[box,fill=yellow] at (8,10){};
\node[box,fill=red] at (9,10){};
\node[box,fill=blue] at (10,10){};
\node[box,fill=blue] at (0,9){};
\node[box,fill=blue] at (1,9){};
\node[box,fill=yellow] at (2,9){};
\node[box,fill=red] at (3,9){};
\node[box,fill=blue] at (4,9){};
\node[box,fill=yellow] at (5,9){};
\node[box,fill=red] at (6,9){};
\node[box,fill=blue] at (7,9){};
\node[box,fill=yellow] at (8,9){};
\node[box,fill=red] at (9,9){};
\node[box,fill=blue] at (10,9){};
\node[box,fill=yellow] at (0,8){};
\node[box,fill=yellow] at (1,8){};
\node[box,fill=yellow] at (2,8){};
\node[box,fill=red] at (3,8){};
\node[box,fill=blue] at (4,8){};
\node[box,fill=yellow] at (5,8){};
\node[box,fill=red] at (6,8){};
\node[box,fill=blue] at (7,8){};
\node[box,fill=yellow] at (8,8){};
\node[box,fill=red] at (9,8){};
\node[box,fill=blue] at (10,8){};
\node[box,fill=red] at (0,7){};
\node[box,fill=red] at (1,7){};
\node[box,fill=red] at (2,7){};
\node[box,fill=red] at (3,7){};
\node[box,fill=blue] at (4,7){};
\node[box,fill=yellow] at (5,7){};
\node[box,fill=red] at (6,7){};
\node[box,fill=blue] at (7,7){};
\node[box,fill=yellow] at (8,7){};
\node[box,fill=red] at (9,7){};
\node[box,fill=blue] at (10,7){};
\node[box,fill=blue] at (0,6){};
\node[box,fill=blue] at (1,6){};
\node[box,fill=blue] at (2,6){};
\node[box,fill=blue] at (3,6){};
\node[box,fill=blue] at (4,6){};
\node[box,fill=yellow] at (5,6){};
\node[box,fill=red] at (6,6){};
\node[box,fill=blue] at (7,6){};
\node[box,fill=yellow] at (8,6){};
\node[box,fill=red] at (9,6){};
\node[box,fill=blue] at (10,6){};
\node[box,fill=yellow] at (0,5){};
\node[box,fill=yellow] at (1,5){};
\node[box,fill=yellow] at (2,5){};
\node[box,fill=yellow] at (3,5){};
\node[box,fill=yellow] at (4,5){};
\node[box,fill=yellow] at (5,5){};
\node[box,fill=red] at (6,5){};
\node[box,fill=blue] at (7,5){};
\node[box,fill=yellow] at (8,5){};
\node[box,fill=red] at (9,5){};
\node[box,fill=blue] at (10,5){};
\node[box,fill=red] at (0,4){};
\node[box,fill=red] at (1,4){};
\node[box,fill=red] at (2,4){};
\node[box,fill=red] at (3,4){};
\node[box,fill=red] at (4,4){};
\node[box,fill=red] at (5,4){};
\node[box,fill=red] at (6,4){};
\node[box,fill=blue] at (7,4){};
\node[box,fill=yellow] at (8,4){};
\node[box,fill=red] at (9,4){};
\node[box,fill=blue] at (10,4){};
\node[box,fill=blue] at (0,3){};
\node[box,fill=blue] at (1,3){};
\node[box,fill=blue] at (2,3){};
\node[box,fill=blue] at (3,3){};
\node[box,fill=blue] at (4,3){};
\node[box,fill=blue] at (5,3){};
\node[box,fill=blue] at (6,3){};
\node[box,fill=blue] at (7,3){};
\node[box,fill=yellow] at (8,3){};
\node[box,fill=red] at (9,3){};
\node[box,fill=blue] at (10,3){};
\node[box,fill=yellow] at (0,2){};
\node[box,fill=yellow] at (1,2){};
\node[box,fill=yellow] at (2,2){};
\node[box,fill=yellow] at (3,2){};
\node[box,fill=yellow] at (4,2){};
\node[box,fill=yellow] at (5,2){};
\node[box,fill=yellow] at (6,2){};
\node[box,fill=yellow] at (7,2){};
\node[box,fill=yellow] at (8,2){};
\node[box,fill=red] at (9,2){};
\node[box,fill=blue] at (10,2){};
\node[box,fill=red] at (0,1){};
\node[box,fill=red] at (1,1){};
\node[box,fill=red] at (2,1){};
\node[box,fill=red] at (3,1){};
\node[box,fill=red] at (4,1){};
\node[box,fill=red] at (5,1){};
\node[box,fill=red] at (6,1){};
\node[box,fill=red] at (7,1){};
\node[box,fill=red] at (8,1){};
\node[box,fill=red] at (9,1){};
\node[box,fill=blue] at (10,1){};
\node[box,fill=blue] at (0,0){};
\node[box,fill=blue] at (1,0){};
\node[box,fill=blue] at (2,0){};
\node[box,fill=blue] at (3,0){};
\node[box,fill=blue] at (4,0){};
\node[box,fill=blue] at (5,0){};
\node[box,fill=blue] at (6,0){};
\node[box,fill=blue] at (7,0){};
\node[box,fill=blue] at (8,0){};
\node[box,fill=blue] at (9,0){};
\node[box,fill=blue] at (10,0){};
\draw[white] (0,10) node {0};
\draw[white] (1,10) node {1};
\draw[gray] (2,10) node {2};
\draw[white] (3,10) node {0};
\draw[white] (4,10) node {1};
\draw[gray] (5,10) node {2};
\draw[white] (6,10) node {0};
\draw[white] (7,10) node {1};
\draw[gray] (8,10) node {2};
\draw[white] (9,10) node {0};
\draw[white] (10,10) node {1};
\draw[white] (0,9) node {1};
\draw[white] (1,9) node {1};
\draw[gray] (2,9) node {2};
\draw[white] (3,9) node {0};
\draw[white] (4,9) node {1};
\draw[gray] (5,9) node {2};
\draw[white] (6,9) node {0};
\draw[white] (7,9) node {1};
\draw[gray] (8,9) node {2};
\draw[white] (9,9) node {0};
\draw[white] (10,9) node {1};
\draw[gray] (0,8) node {2};
\draw[gray] (1,8) node {2};
\draw[gray] (2,8) node {2};
\draw[white] (3,8) node {0};
\draw[white] (4,8) node {1};
\draw[gray] (5,8) node {2};
\draw[white] (6,8) node {0};
\draw[white] (7,8) node {1};
\draw[gray] (8,8) node {2};
\draw[white] (9,8) node {0};
\draw[white] (10,8) node {1};
\draw[white] (0,7) node {0};
\draw[white] (1,7) node {0};
\draw[white] (2,7) node {0};
\draw[white] (3,7) node {0};
\draw[white] (4,7) node {1};
\draw[gray] (5,7) node {2};
\draw[white] (6,7) node {0};
\draw[white] (7,7) node {1};
\draw[gray] (8,7) node {2};
\draw[white] (9,7) node {0};
\draw[white] (10,7) node {1};
\draw[white] (0,6) node {1};
\draw[white] (1,6) node {1};
\draw[white] (2,6) node {1};
\draw[white] (3,6) node {1};
\draw[white] (4,6) node {1};
\draw [gray](5,6) node {2};
\draw[white] (6,6) node {0};
\draw[white] (7,6) node {1};
\draw [gray](8,6) node {2};
\draw[white] (9,6) node {0};
\draw[white] (10,6) node {1};
\draw[gray] (0,5) node {2};
\draw [gray](1,5) node {2};
\draw [gray](2,5) node {2};
\draw [gray](3,5) node {2};
\draw [gray](4,5) node {2};
\draw [gray](5,5) node {2};
\draw[white] (6,5) node {0};
\draw[white] (7,5) node {1};
\draw [gray](8,5) node {2};
\draw[white] (9,5) node {0};
\draw[white] (10,5) node {1};
\draw[white] (0,4) node {0};
\draw[white] (1,4) node {0};
\draw[white] (2,4) node {0};
\draw[white] (3,4) node {0};
\draw[white] (4,4) node {0};
\draw[white] (5,4) node {0};
\draw[white] (6,4) node {0};
\draw[white] (7,4) node {1};
\draw[gray](8,4) node {2};
\draw[white] (9,4) node {0};
\draw[white] (10,4) node {1};
\draw[white] (0,3) node {1};
\draw[white] (1,3) node {1};
\draw[white] (2,3) node {1};
\draw[white] (3,3) node {1};
\draw[white] (4,3) node {1};
\draw[white] (5,3) node {1};
\draw[white] (6,3) node {1};
\draw[white] (7,3) node {1};
\draw[gray] (8,3) node {2};
\draw[white] (9,3) node {0};
\draw[white] (10,3) node {1};
\draw[gray] (0,2) node {2};
\draw[gray] (1,2) node {2};
\draw[gray] (2,2) node {2};
\draw[gray] (3,2) node {2};
\draw[gray] (4,2) node {2};
\draw[gray] (5,2) node {2};
\draw[gray] (6,2) node {2};
\draw[gray] (7,2) node {2};
\draw[gray] (8,2) node {2};
\draw[white] (9,2) node {0};
\draw[white] (10,2) node {1};
\draw[white] (0,1) node {0};
\draw[white] (1,1) node {0};
\draw[white] (2,1) node {0};
\draw[white] (3,1) node {0};
\draw[white] (4,1) node {0};
\draw[white] (5,1) node {0};
\draw[white] (6,1) node {0};
\draw[white] (7,1) node {0};
\draw[white] (8,1) node {0};
\draw[white] (9,1) node {0};
\draw[white] (10,1) node {1};
\draw[white] (0,0) node {1};
\draw[white] (1,0) node {1};
\draw[white] (2,0) node {1};
\draw[white] (3,0) node {1};
\draw[white] (4,0) node {1};
\draw[white] (5,0) node {1};
\draw[white] (6,0) node {1};
\draw[white] (7,0) node {1};
\draw[white] (8,0) node {1};
\draw[white] (9,0) node {1};
\draw[white] (10,0) node {1};
\end{tikzpicture}
\vspace{3mm}
\caption{The initial nimbers for the rulesets $C$ (left) and $D$ (right).}
\label{fig:CandD}
\end{figure}
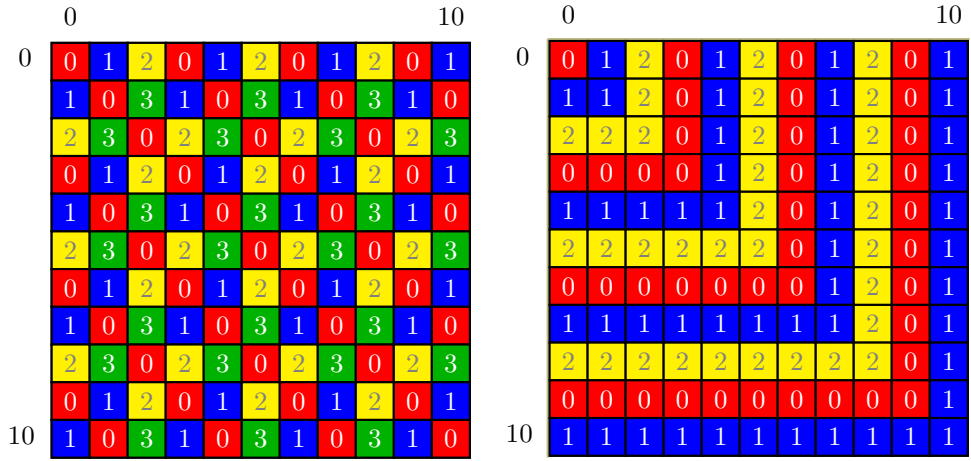
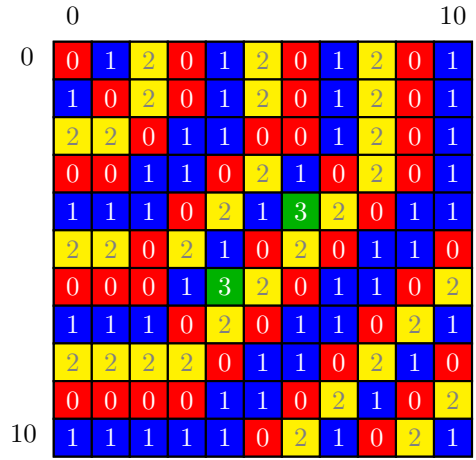
\begin{figure}[htbp!]
    \centering
\begin{tikzpicture}[scale = 0.5,
         box/.style={rectangle,draw=black, thick, minimum size=.5 cm},
     ]
 \foreach \x in {0,1,...,10}{
     \foreach \y in {0,1,...,10}
         \node[box] at (\x,\y){};
 }
\draw (0, 11.2) node {$0$};
\draw (10, 11.2) node {$10$};
\draw (-1.2, 10.1) node {$0$};
\draw (-1.3, 0.1) node {$10$};
\node[box,fill=red] at (0,10){};
\node[box,fill=blue] at (1,10){};
\node[box,fill=yellow] at (2,10){};
\node[box,fill=red] at (3,10){};
\node[box,fill=blue] at (4,10){};
\node[box,fill=yellow] at (5,10){};
\node[box,fill=red] at (6,10){};
\node[box,fill=blue] at (7,10){};
\node[box,fill=yellow] at (8,10){};
\node[box,fill=red] at (9,10){};
\node[box,fill=blue] at (10,10){};
\node[box,fill=blue] at (0,9){};
\node[box,fill=red] at (1,9){};
\node[box,fill=yellow] at (2,9){};
\node[box,fill=red] at (3,9){};
\node[box,fill=blue] at (4,9){};
\node[box,fill=yellow] at (5,9){};
\node[box,fill=red] at (6,9){};
\node[box,fill=blue] at (7,9){};
\node[box,fill=yellow] at (8,9){};
\node[box,fill=red] at (9,9){};
\node[box,fill=blue] at (10,9){};
\node[box,fill=yellow] at (0,8){};
\node[box,fill=yellow] at (1,8){};
\node[box,fill=red] at (2,8){};
\node[box,fill=blue] at (3,8){};
\node[box,fill=blue] at (4,8){};
\node[box,fill=red] at (5,8){};
\node[box,fill=red] at (6,8){};
\node[box,fill=blue] at (7,8){};
\node[box,fill=yellow] at (8,8){};
\node[box,fill=red] at (9,8){};
\node[box,fill=blue] at (10,8){};
\node[box,fill=red] at (0,7){};
\node[box,fill=red] at (1,7){};
\node[box,fill=blue] at (2,7){};
\node[box,fill=blue] at (3,7){};
\node[box,fill=red] at (4,7){};
\node[box,fill=yellow] at (5,7){};
\node[box,fill=blue] at (6,7) {};
\node[box,fill=red] at (7,7){};
\node[box,fill=yellow] at (8,7){};
\node[box,fill=red] at (9,7){};
\node[box,fill=blue] at (10,7){};
\node[box,fill=blue] at (0,6){};
\node[box,fill=blue] at (1,6){};
\node[box,fill=blue] at (2,6){};
\node[box,fill=red] at (3,6){};
\node[box,fill=yellow] at (4,6){};
\node[box,fill=blue] at (5,6){};
\node[box,fill=teal] at (6,6) {};
\node[box,fill=yellow] at (7,6){};
\node[box,fill=red] at (8,6){};
\node[box,fill=blue] at (9,6){};
\node[box,fill=blue] at (10,6){};
\node[box,fill=yellow] at (0,5){};
\node[box,fill=yellow] at (1,5){};
\node[box,fill=red] at (2,5){};
\node[box,fill=yellow] at (3,5){};
\node[box,fill=blue] at (4,5){};
\node[box,fill=red] at (5,5){};
\node[box,fill=yellow] at (6,5) {};
\node[box,fill=red] at (7,5){};
\node[box,fill=blue] at (8,5){};
\node[box,fill=blue] at (9,5){};
\node[box,fill=red] at (10,5){};
\node[box,fill=red] at (0,4){};
\node[box,fill=red] at (1,4){};
\node[box,fill=red] at (2,4){};
\node[box,fill=blue] at (3,4){};
\node[box,fill=teal] at (4,4){};
\node[box,fill=yellow] at (5,4){};
\node[box,fill=red] at (6,4) {};
\node[box,fill=blue] at (7,4){};
\node[box,fill=blue] at (8,4){};
\node[box,fill=red] at (9,4){};
\node[box,fill=yellow] at (10,4){};
\node[box,fill=blue] at (0,3){};
\node[box,fill=blue] at (1,3){};
\node[box,fill=blue] at (2,3){};
\node[box,fill=red] at (3,3){};
\node[box,fill=yellow] at (4,3){};
\node[box,fill=red] at (5,3){};
\node[box,fill=blue] at (6,3) {};
\node[box,fill=blue] at (7,3){};
\node[box,fill=red] at (8,3){};
\node[box,fill=yellow] at (9,3){};
\node[box,fill=blue] at (10,3){};
\node[box,fill=yellow] at (0,2){};
\node[box,fill=yellow] at (1,2){};
\node[box,fill=yellow] at (2,2){};
\node[box,fill=yellow] at (3,2){};
\node[box,fill=red] at (4,2){};
\node[box,fill=blue] at (5,2){};
\node[box,fill=blue] at (6,2) {};
\node[box,fill=red] at (7,2){};
\node[box,fill=yellow] at (8,2){};
\node[box,fill=blue] at (9,2){};
\node[box,fill=red] at (10,2){};
\node[box,fill=red] at (0,1){};
\node[box,fill=red] at (1,1){};
\node[box,fill=red] at (2,1){};
\node[box,fill=red] at (3,1){};
\node[box,fill=blue] at (4,1){};
\node[box,fill=blue] at (5,1){};
\node[box,fill=red] at (6,1) {};
\node[box,fill=yellow] at (7,1){};
\node[box,fill=blue] at (8,1){};
\node[box,fill=red] at (9,1){};
\node[box,fill=yellow] at (10,1){};
\node[box,fill=blue] at (0,0){};
\node[box,fill=blue] at (1,0){};
\node[box,fill=blue] at (2,0){};
\node[box,fill=blue] at (3,0){};
\node[box,fill=blue] at (4,0){};
\node[box,fill=red] at (5,0){};
\node[box,fill=yellow] at (6,0) {};
\node[box,fill=blue] at (7,0){};
\node[box,fill=red] at (8,0){};
\node[box,fill=yellow] at (9,0){};
\node[box,fill=blue] at (10,0){};
\draw[white] (0,10) node {0};
\draw[white] (1,10) node {1};
\draw[gray] (2,10) node {2};
\draw[white] (3,10) node {0};
\draw[white] (4,10) node {1};
\draw[gray] (5,10) node {2};
\draw[white] (6,10) node {0};
\draw[white] (7,10) node {1};
\draw[gray] (8,10) node {2};
\draw[white] (9,10) node {0};
\draw[white] (10,10) node {1};
\draw[white] (0,9) node {1};
\draw[white] (1,9) node {0};
\draw[gray] (2,9) node {2};
\draw[white] (3,9) node {0};
\draw[white] (4,9) node {1};
\draw[gray] (5,9) node {2};
\draw[white] (6,9) node {0};
\draw[white] (7,9) node {1};
\draw[gray] (8,9) node {2};
\draw[white] (9,9) node {0};
\draw[white] (10,9) node {1};
\draw[gray] (0,8) node {2};
\draw[gray] (1,8) node {2};
\draw[white] (2,8) node {0};
\draw[white] (3,8) node {1};
\draw[white] (4,8) node {1};
\draw[white] (5,8) node {0};
\draw[white] (6,8) node {0};
\draw[white] (7,8) node {1};
\draw[gray] (8,8) node {2};
\draw[white] (9,8) node {0};
\draw[white] (10,8) node {1};
\draw[white] (0,7) node {0};
\draw[white] (1,7) node {0};
\draw[white] (2,7) node {1};
\draw[white] (3,7) node {1};
\draw[white] (4,7) node {0};
\draw[gray] (5,7) node {2};
\draw[white] (6,7) node {1};
\draw[white] (7,7) node {0};
\draw[gray] (8,7) node {2};
\draw[white] (9,7) node {0};
\draw[white] (10,7) node {1};
\draw[white] (0,6) node {1};
\draw[white] (1,6) node {1};
\draw[white] (2,6) node {1};
\draw[white] (3,6) node {0};
\draw[gray]  (4,6) node {2};
\draw[white] (5,6) node {1};
\draw[white] (6,6) node {3};
\draw[gray] (7,6) node {2};
\draw[white] (8,6) node {0};
\draw[white] (9,6) node {1};
\draw[white] (10,6) node {1};
\draw[gray] (0,5) node {2};
\draw[gray] (1,5) node {2};
\draw[white] (2,5) node {0};
\draw[gray] (3,5) node {2};
\draw[white] (4,5) node {1};
\draw[white] (5,5) node {0};
\draw[gray] (6,5) node {2};
\draw[white] (7,5) node {0};
\draw[white] (8,5) node {1};
\draw[white] (9,5) node {1};
\draw[white] (10,5) node {0};
\draw[white] (0,4) node {0};
\draw[white] (1,4) node {0};
\draw[white] (2,4) node {0};
\draw[white] (3,4) node {1};
\draw[white] (4,4) node {3};
\draw[gray] (5,4) node {2};
\draw[white] (6,4) node {0};
\draw[white] (7,4) node {1};
\draw[white] (8,4) node {1};
\draw[white] (9,4) node {0};
\draw[gray] (10,4) node {2};
\draw[white] (0,3) node {1};
\draw[white] (1,3) node {1};
\draw[white] (2,3) node {1};
\draw[white] (3,3) node {0};
\draw[gray] (4,3) node {2};
\draw[white] (5,3) node {0};
\draw[white] (6,3) node {1};
\draw[white] (7,3) node {1};
\draw[white] (8,3) node {0};
\draw[gray] (9,3) node {2};
\draw[white] (10,3) node {1};
\draw[gray] (0,2) node {2};
\draw[gray] (1,2) node {2};
\draw[gray] (2,2) node {2};
\draw[gray] (3,2) node {2};
\draw[white] (4,2) node {0};
\draw[white] (5,2) node {1};
\draw[white] (6,2) node {1};
\draw[white] (7,2) node {0};
\draw[gray] (8,2) node {2};
\draw[white] (9,2) node {1};
\draw[white] (10,2) node {0};
\draw[white] (0,1) node {0};
\draw[white] (1,1) node {0};
\draw[white] (2,1) node {0};
\draw[white] (3,1) node {0};
\draw[white] (4,1) node {1};
\draw[white] (5,1) node {1};
\draw[white] (6,1) node {0};
\draw[gray] (7,1) node {2};
\draw[white] (8,1) node {1};
\draw[white] (9,1) node {0};
\draw[gray] (10,1) node {2};
\draw[white] (0,0) node {1};
\draw[white] (1,0) node {1};
\draw[white] (2,0) node {1};
\draw[white] (3,0) node {1};
\draw[white] (4,0) node {1};
\draw[white] (5,0) node {0};
\draw[gray] (6,0) node {2};
\draw[white] (7,0) node {1};
\draw[white] (8,0) node {0};
\draw[gray] (9,0) node {2};
\draw[white] (10,0) node {1};
\end{tikzpicture}
\vspace{3mm}
\caption{The initial nimbers for $C \odot D$.}
\label{fig:CenforceD}
\end{figure}

\begin{example}
\label{ex:nimvalbishopknight}
Let $\X=\mathbb{Z}_{\geq 0} \times \mathbb{Z}_{\geq 0}$ and consider the  rulesets  {\sc White Bishop} $\large \symbishop$ and {\sc White Knight} $\large\symknight$ \cite{BCG82}, defined by 
$f_{\symbishop}((x, y))=\{(x-i,y-i) \mid 1 \leq i \leq \min(x,y) \}$ and $f_{\symknight}((x, y))=\{(x-2,y+1), (x-2,y-1), (x-1,y-2), (x+1, y-2) \}$. Figure~\ref{fig:WBWK2HN} displays the initial nimbers initial for these rulesets,  while the enforce combination  $\large \symbishop\odot \large \symknight $, is displayed in Figure~\ref{fig:BenfR_KenfR_BenfK}. Observe that the rulesets $\large\symbishop $ and $\large\symknight$ are confused.\\
\end{example}

\begin{figure}[htbp!]
\begin{center}
\begin{tikzpicture}[scale = 0.5,
         box/.style={rectangle,draw=black, thick, minimum size=.5 cm},
     ]

\foreach \x in {0,1,...,5}{
     \foreach \y in {0,1,...,5}
         \node[box] at (\x,\y){};
 }

\draw (0, 6.2) node {$0$};
\draw (5, 6.2) node {$5$};
 \draw (-1.2, 5.1) node {$0$};
 \draw (-1.2, 0.1) node {$5$};
 \node[box,fill=red] at (0,0){};
 \node[box,fill=red] at (0,1){};
 \node[box,fill=red] at (0,2){};
 \node[box,fill=red] at (0,3){};
 \node[box,fill=red] at (0,4){};
 \node[box,fill=red] at (0,5){};
\node[box,fill=red] at (1,5){};
\node[box,fill=red] at (2,5){};
\node[box,fill=red] at (3,5){};
\node[box,fill=red] at (4,5){};
\node[box,fill=red] at (5,5){};
\node[box,fill=blue] at (1,4){};
\node[box,fill=blue] at (2,4){};
\node[box,fill=blue] at (3,4){};
\node[box,fill=blue] at (4,4){};
\node[box,fill=blue] at (5,4){};
\node[box,fill=blue] at (1,3){};
\node[box,fill=yellow] at (2,3){};
\node[box,fill=yellow] at (3,3){};
\node[box,fill=yellow] at (4,3){};
\node[box,fill=yellow] at (5,3){};
\node[box,fill=blue] at (1,2){};
\node[box,fill=yellow] at (2,2){};
\node[box,fill=teal] at (3,2){};
\node[box,fill=teal] at (4,2){};
\node[box,fill=teal] at (5,2){};
\node[box,fill=blue] at (1,1){};
\node[box,fill=yellow] at (2,1){};
\node[box,fill=teal] at (3,1){};
\node[box,fill=green] at (4,1){};
\node[box,fill=green] at (5,1){};
\node[box,fill= blue] at (1,0){};
\node[box,fill=yellow] at (2,0){};
\node[box,fill=teal] at (3,0){};
\node[box,fill=green] at (4,0){};
\node[box,fill=brown] at (5,0){};
\draw[white] (0,0) node {0};
\draw[white] (0,1) node {0};
\draw[white] (0,2) node {0};
\draw[white] (0,3) node {0};
 \draw[white] (0,4) node {0};
 \draw[white] (0,5) node {0};
\draw[white] (1,5) node {0};
\draw[white] (2,5) node {0};
\draw[white] (3,5) node {0};
\draw[white] (4,5) node {0};
\draw[white] (5,5) node {0};
\draw[white] (1,4) node {1};
\draw[white] (2,4) node {1};
\draw[white] (3,4) node {1};
\draw[white] (4,4) node {1};
\draw[white] (5,4) node {1};
\draw[white] (1,3) node {1 };
\draw[gray] (2,3) node {2};
\draw[gray] (3,3) node {2};
\draw[gray] (4,3) node {2};
\draw[gray] (5,3) node {2};
\draw[white] (1,2) node {1};
\draw[gray] (2,2) node {2};
\draw[white] (3,2) node {3};
\draw[white] (4,2) node {3};
\draw[white] (5,2) node {3};
\draw[white] (1,1) node {1};
\draw[gray] (2,1) node {2};
\draw[white] (3,1) node {3};
\draw[white] (4,1) node {4};
\draw[white] (5,1) node {4};
\draw[white]  (1,0) node {1};
\draw[gray] (2,0) node {2};
\draw[white] (3,0) node {3};
\draw[white] (4,0) node {4};
\draw[white] (5,0) node {5};
\end{tikzpicture}
\begin{tikzpicture}[scale = 0.5,
         box/.style={rectangle,draw=black, thick, minimum size=.5 cm},
     ]

\foreach \x in {0,1,...,5}{
     \foreach \y in {0,1,...,5}
         \node[box] at (\x,\y){};
 }

\draw[thick, color=gray!60!yellow] (-0.55,-0.5) -- (-0.55, 5.55);

\draw[thick, color=gray!60!yellow] (-0.55, 5.55)-- (5.55, 5.55);

\draw (0, 6.2) node {$0$};
\draw (5, 6.2) node {$5$};

 \draw (-1.2, 5.1) node {$0$};
 \draw (-1.2, 0.1) node {$5$};

 \node[box,fill=red] at (0,0){};
 \node[box,fill=red] at (0,1){};
 \node[box,fill=blue] at (0,2){};
 \node[box,fill=blue] at (0,3){};
 \node[box,fill=red] at (0,4){};
 \node[box,fill=red] at (0,5){};

\node[box,fill=red] at (1,5){};
\node[box,fill=blue] at (2,5){};
\node[box,fill=blue] at (3,5){};
\node[box,fill=red] at (4,5){};
\node[box,fill=red] at (5,5){};

\node[box,fill=red] at (1,4){};
\node[box,fill=yellow] at (2,4){};
\node[box,fill=blue] at (3,4){};
\node[box,fill=red] at (4,4){};
\node[box,fill=red] at (5,4){};

\node[box,fill=yellow] at (1,3){};
\node[box,fill=yellow] at (2,3){};
\node[box,fill=yellow] at (3,3){};
\node[box,fill=teal] at (4,3){};
\node[box,fill=yellow] at (5,3){};

\node[box,fill=blue] at (1,2){};
\node[box,fill=yellow] at (2,2){};
\node[box,fill=blue] at (3,2){};
\node[box,fill=green] at (4,2){};
\node[box,fill=teal] at (5,2){};

\node[box,fill=red] at (1,1){};
\node[box,fill=teal] at (2,1){};
\node[box,fill=green] at (3,1){};
\node[box,fill=red] at (4,1){};
\node[box,fill=red] at (5,1){};

\node[box,fill= red] at (1,0){};
\node[box,fill=yellow] at (2,0){};
\node[box,fill=teal] at (3,0){};
\node[box,fill=red] at (4,0){};
\node[box,fill=red] at (5,0){};

\draw[white] (0,0) node {0};
\draw[white] (0,1) node {0};
\draw[white] (0,2) node {1};
\draw[white] (0,3) node {1};
\draw[white] (0,4) node {0};
 \draw[white] (0,5) node {0};

\draw[white] (1,5) node {0};
\draw[white] (2,5) node {1};
\draw[white] (3,5) node {1};
\draw[white] (4,5) node  {0};
\draw[white] (5,5) node {0};

\draw[white] (1,4) node {0};
\draw[gray] (2,4) node {2};
\draw[white] (3,4) node {1};
\draw[white] (4,4) node {0};
\draw[white] (5,4) node {0};

\draw[gray] (1,3) node {2};
\draw[gray] (2,3) node {2};
\draw[gray] (3,3) node {2};
\draw[white] (4,3) node {3};
\draw[gray] (5,3) node {2};

\draw[white] (1,2) node {1};
\draw[gray] (2,2) node {2};
\draw[white] (3,2) node {1};
\draw[white] (4,2) node {4};
\draw[white] (5,2) node {3};

\draw[white] (1,1) node {0};
\draw[white] (2,1) node {3};
\draw[white] (3,1) node {4};
\draw[white] (4,1) node {0};
\draw[white] (5,1) node {0};

\draw[white] (1,0) node {0};
\draw[gray] (2,0) node {2};
\draw[white] (3,0) node {3};
\draw[white] (4,0) node {0};
\draw[white] (5,0) node {0};
\end{tikzpicture}
\begin{tikzpicture}[scale = 0.5,
         box/.style={rectangle,draw=black, thick, minimum size=.5 cm},
     ]

\foreach \x in {0,1,...,5}{
     \foreach \y in {0,1,...,5}
         \node[box] at (\x,\y){};
 }



\draw (0, 6.2) node {$0$};
\draw (5, 6.2) node {$5$};

\draw (-1.2, 5.1) node {$0$};
\draw (-1.2, 0.1) node {$5$};

\node[box,fill=brown] at (0,0){};
\node[box,fill=green] at (0,1){};
\node[box,fill=teal] at (0,2){};
\node[box,fill=yellow] at (0,3){};
\node[box,fill=blue] at (0,4){};
\node[box,fill=red] at (0,5){};

\draw (1,0) node {4};
\draw (1,1) node {5};
\draw (1,2) node {2};
\draw (1,3) node {3};
\draw (1,4) node {0};
\draw (1,5) node {1};
\node[box,fill=green] at (1,0){};
\node[box,fill=brown] at (1,1){};
\node[box,fill=yellow] at (1,2){};
\node[box,fill=teal] at (1,3){};
\node[box,fill=red] at (1,4){};
\node[box,fill=blue] at (1,5){};

\node[box,fill=gray] at (2,0){};
\node[box,fill=purple] at (2,1){};
\node[box,fill=blue] at (2,2){};
\node[box,fill=red] at (2,3){};
\node[box,fill=teal] at (2,4){};
\node[box,fill=yellow] at (2,5){};

\node[box,fill=purple] at (3,0){};
\node[box,fill=gray] at (3,1){};
\node[box,fill=red] at (3,2){};
\node[box,fill=blue] at (3,3){};
\node[box,fill=yellow] at (3,4){};
\node[box,fill=teal] at (3,5){};

\draw[white] (4,0) node {1};
\draw[white] (4,1) node {0};
\draw[white] (4,2) node {7};
\draw[white] (4,3) node {6};
\draw[white] (4,4) node {5};
\draw[white] (4,5) node {4};
\node[box,fill=blue] at (4,0){};
\node[box,fill=red] at (4,1){};
\node[box,fill=gray] at (4,2){};
\node[box,fill=purple] at (4,3){};
\node[box,fill=brown] at (4,4){};
\node[box,fill=green] at (4,5){};

\draw[white] (5,0) node {0};
\draw[white] (5,1) node {1};
\draw[white] (5,2) node {6};
\draw[white] (5,3) node {7};
\draw[white] (5,4) node {4};
\draw[white] (5,5) node {5};
\node[box,fill=red] at (5,0){};
\node[box,fill=blue] at (5,1){};
\node[box,fill=purple] at (5,2){};
\node[box,fill=gray] at (5,3){};
\node[box,fill=green] at (5,4){};
\node[box,fill=brown] at (5,5){};
\draw[white] (0,0) node{5};
\draw[white] (0,1) node {4};
\draw[white] (0,2) node {3};
\draw[gray] (0,3) node {2};
\draw[white] (0,4) node {1};
\draw[white] (0,5) node {0};
\draw[white] (1,0) node {4};
\draw[white] (1,1) node {5};
\draw[gray] (1,2) node {2};
\draw[white] (1,3) node {3};
\draw[white] (1,4) node {0};
\draw[white] (1,5) node {1};

\draw[white] (2,0) node {7};
\draw[white] (2,1) node {6};
\draw[white] (2,2) node {1};
\draw[white] (2,3) node {0};
\draw[white] (2,4) node {3};
\draw[gray] (2,5) node {2};

\draw[white] (3,0) node {6};
\draw[white] (3,1) node {7};
\draw[white] (3,2) node {0};
\draw[white] (3,3) node {1};
\draw[gray] (3,4) node {2};
\draw[white] (3,5) node {3};

\draw[white] (4,0) node {1};
\draw[white] (4,1) node {0};
\draw[white] (4,2) node {7};
\draw[white] (4,3) node {6};
\draw[white] (4,4) node {5};
\draw[white] (4,5) node {4};

\draw[white] (5,0) node {0};
\draw[white] (5,1) node {1};
\draw[white] (5,2) node {6};
\draw[white] (5,3) node {7};
\draw[white] (5,4) node {4};
\draw[white] (5,5) node {5};
\end{tikzpicture}
\end{center}
 \vspace{3mm}
\caption{The initial nimbers of {\sc White Bishop} $\large \symbishop$ (left) {\sc White Knight}, $\large \symknight$ (middle), and 
 {\sc Two Heap Nim} $\large \symrook$ (right).}\label{fig:WBWK2HN}
\end{figure}

\begin{figure}[htbp!]
\begin{center}
\begin{tikzpicture}[scale = 0.5,
         box/.style={rectangle,draw=black, thick, minimum size=.5 cm},
     ]
\foreach \x in {0,1,...,5}{
     \foreach \y in {0,1,...,5}
         \node[box] at (\x,\y){};
 }
\draw (0, 6.2) node {$0$};
\draw (5, 6.2) node {$5$};
 \draw (-1.2, 5.1) node {$0$};
 \draw (-1.2, 0.1) node {$5$};
 \node[box,fill=red] at (0,0){};
 \node[box,fill=red] at (0,1){};
 \node[box,fill=red] at (0,2){};
 \node[box,fill=red] at (0,3){};
 \node[box,fill=red] at (0,4){};
 \node[box,fill=red] at (0,5){};
\node[box,fill=red] at (1,5){};
\node[box,fill=red] at (2,5){};
\node[box,fill=red] at (3,5){};
\node[box,fill=red] at (4,5){};
\node[box,fill=red] at (5,5){};
\node[box,fill=blue] at (1,4){};
\node[box,fill=blue] at (2,4){};
\node[box,fill=blue] at (3,4){};
\node[box,fill=blue] at (4,4){};
\node[box,fill=blue] at (5,4){};
\node[box,fill=blue] at (1,3){};
\node[box,fill=yellow] at (2,3){};
\node[box,fill=yellow] at (3,3){};
\node[box,fill=yellow] at (4,3){};
\node[box,fill=yellow] at (5,3){};
\node[box,fill=blue] at (1,2){};
\node[box,fill=yellow] at (2,2){};
\node[box,fill=teal] at (3,2){};
\node[box,fill=teal] at (4,2){};
\node[box,fill=teal] at (5,2){};
\node[box,fill=blue] at (1,1){};
\node[box,fill=yellow] at (2,1){};
\node[box,fill=teal] at (3,1){};
\node[box,fill=green] at (4,1){};
\node[box,fill=green] at (5,1){};
\node[box,fill= blue] at (1,0){};
\node[box,fill=yellow] at (2,0){};
\node[box,fill=teal] at (3,0){};
\node[box,fill=green] at (4,0){};
\node[box,fill=brown] at (5,0){};
\draw[white] (0,0) node {0};
\draw[white] (0,1) node {0};
\draw[white] (0,2) node {0};
\draw[white] (0,3) node {0};
\draw[white] (0,4) node {0};
\draw[white] (0,5) node {0};
\draw[white] (1,5) node {0};
\draw[white] (2,5) node {0};
\draw[white] (3,5) node {0};
\draw[white] (4,5) node {0};
\draw[white] (5,5) node {0};
\draw[white] (1,4) node {1};
\draw[white] (2,4) node {1};
\draw[white] (3,4) node {1};
\draw[white] (4,4) node {1};
\draw[white] (5,4) node {1};
\draw[white] (1,3) node {1};
\draw[gray] (2,3) node {2};
\draw[gray] (3,3) node {2};
\draw[gray] (4,3) node {2};
\draw[gray] (5,3) node {2};
\draw[white] (1,2) node {1};
\draw[gray] (2,2) node {2};
\draw[white] (3,2) node {3};
\draw[white] (4,2) node {3};
\draw[white] (5,2) node {3};
\draw[white] (1,1) node {1};
\draw[gray] (2,1) node {2};
\draw[white] (3,1) node {3};
\draw[white] (4,1) node {4};
\draw[white] (5,1) node {4};
\draw[white] (1,0) node {1};
\draw[gray] (2,0) node {2};
\draw[white] (3,0) node {3};
\draw[white] (4,0) node {4};
\draw[white] (5,0) node {5};
\end{tikzpicture}\hspace{2 mm}
\begin{tikzpicture}[scale = 0.5,
         box/.style={rectangle,draw=black, thick, minimum size=.5 cm},
     ]
\foreach \x in {0,1,...,5}{
     \foreach \y in {0,1,...,5}
         \node[box] at (\x,\y){};
 }
\draw (0, 6.2) node {$0$};
\draw (5, 6.2) node {$5$};
\draw (-1.2, 5.1) node {$0$};
\draw (-1.2, 0.1) node {$5$};
\draw[white] (0,0) node {0};
\draw[white] (0,1) node {0};
\draw[white] (0,2) node {1};
\draw[white] (0,3) node {1};
\draw[white] (0,4) node {0};
\draw[white] (0,5) node {0};
\node[box,fill=red] at (0,0){};
\node[box,fill=red] at (0,1){};
\node[box,fill=blue] at (0,2){};
\node[box,fill=blue] at (0,3){};
\node[box,fill=red] at (0,4){};
\node[box,fill=red] at (0,5){};
\draw[white] (5,5) node {0};
\draw[white] (4,5) node {0};
\draw[white] (3,5) node {1};
\draw[white] (2,5) node {1};
\draw[white] (1,5) node {0};
\node[box,fill=red] at (5,5){};
\node[box,fill=red] at (4,5){};
\node[box,fill=blue] at (3,5){};
\node[box,fill=blue] at (2,5){};
\node[box,fill=red] at (1,5){};
\node[box,fill=red] at (1,4){};
\node[box,fill=yellow] at (2,4){};
\node[box,fill=blue] at (3,4){};
\node[box,fill=yellow] at (4,4){};
\node[box,fill=red] at (5,4){};
\node[box,fill=yellow] at (1,3){};
\node[box,fill=red] at (2,3){};
\node[box,fill=yellow] at (3,3){};
\node[box,fill=teal] at (4,3){};
\node[box,fill=yellow] at (5,3){};
\node[box,fill=blue] at (1,2){};
\node[box,fill=yellow] at (2,2){};
\node[box,fill=red] at (3,2){};
\node[box,fill=yellow] at (4,2){};
\node[box,fill=red] at (5,2){};
\node[box,fill=yellow] at (1,1){};
\node[box,fill=teal] at (2,1){};
\node[box,fill=yellow] at (3,1){};
\node[box,fill=red] at (4,1){};
\node[box,fill=blue] at (5,1){};
\node[box,fill=red] at (1,0){};
\node[box,fill=yellow] at (2,0){};
\node[box,fill=red] at (3,0){};
\node[box,fill=blue] at (4,0){};
\node[box,fill=red] at (5,0){};
\draw[white] (0,0) node {0};
\draw[white] (0,1) node {0};
\draw[white] (0,2) node {1};
\draw[white] (0,3) node {1};
\draw[white] (0,4) node {0};
\draw[white] (0,5) node {0};
\draw[white] (5,5) node {0};
\draw[white] (4,5) node {0};
\draw[white] (3,5) node {1};
\draw[white] (2,5) node {1};
\draw[white] (1,5) node {0};
\draw[white]  (1,4) node {0};
\draw[gray] (2,4) node {2};
\draw[white] (3,4) node {1};
\draw[gray] (4,4) node {2};
\draw[white] (5,4) node {0};
\draw[gray] (1,3) node {2};
\draw[white] (2,3) node {0};
\draw[gray] (3,3) node {2};
\draw[white] (4,3) node {3};
\draw[gray] (5,3) node {2};
\draw[white] (1,2) node {1};
\draw[gray] (2,2) node {2};
\draw[white] (3,2) node {0};
\draw[gray] (4,2) node {2};
\draw[white] (5,2) node {0};
\draw[gray] (1,1) node {2};
\draw[white] (2,1) node {3};
\draw[gray] (3,1) node {2};
\draw[white] (4,1) node {0};
\draw[white] (5,1) node {1};
\draw[white] (1,0) node {0};
\draw[gray] (2,0) node {2};
\draw[white] (3,0) node {0};
\draw[white] (4,0) node {1};
\draw[white] (5,0) node {0};
\end{tikzpicture}\hspace{2 mm}
\begin{tikzpicture}[scale = 0.5,
         box/.style={rectangle,draw=black, thick, minimum size=.5 cm},
     ]
\foreach \x in {0,1,...,5}{
     \foreach \y in {0,1,...,5}
         \node[box] at (\x,\y){};
 }
\draw[thick, color=gray!60!yellow] (-0.55,-0.5) -- (-0.55, 5.55);
\draw[thick, color=gray!60!yellow] (-0.55, 5.55)-- (5.55, 5.55);
\draw (0, 6.2) node {$0$};
\draw (5, 6.2) node {$5$};
\draw (-1.2, 5.1) node {$0$};
\draw (-1.2, 0.1) node {$5$};
\node[box,fill=red] at (0,0){};
\node[box,fill=red] at (0,1){};
\node[box,fill=red] at (0,2){};
\node[box,fill=red] at (0,3){};
\node[box,fill=red] at (0,4){};
\node[box,fill=red] at (0,5){};
\node[box,fill=blue] at (1,0){};
\node[box,fill=blue] at (1,1){};
\node[box,fill=blue] at (1,2){};
\node[box,fill=blue] at (1,3){};
\node[box,fill=red] at (1,4){};
\node[box,fill=red] at (1,5){};
\node[box,fill=yellow] at (2,0){};
\node[box,fill=yellow] at (2,1){};
\node[box,fill=yellow] at (2,2){};
\node[box,fill=blue] at (2,3){};
\node[box,fill=blue] at (2,4){};
\node[box,fill=red] at (2,5){};
\node[box,fill=teal] at (3,0){};
\node[box,fill=red] at (3,1){};
\node[box,fill=red] at (3,2){};
\node[box,fill=yellow] at (3,3){};
\node[box,fill=blue] at (3,4){};
\node[box,fill=red] at (3,5){};
\node[box,fill=blue] at (4,0){};
\node[box,fill=red] at (4,1){};
\node[box,fill=red] at (4,2){};
\node[box,fill=yellow] at (4,3){};
\node[box,fill=blue] at (4,4){};
\node[box,fill=red] at (4,5){};
\node[box,fill=blue] at (5,0){};
\node[box,fill=blue] at (5,1){};
\node[box,fill=teal] at (5,2){};
\node[box,fill=yellow] at (5,3){};
\node[box,fill=blue] at (5,4){};
\node[box,fill=red] at (5,5){};
\draw[white] (0,0) node {0};
\draw[white] (0,1) node {0};
\draw[white] (0,2) node {0};
\draw[white] (0,3) node {0};
\draw[white] (0,4) node {0};
\draw[white] (0,5) node {0};
\draw[white] (1,0) node {1};
\draw[white] (1,1) node {1};
\draw[white] (1,2) node {1};
\draw[white] (1,3) node {1};
\draw[white] (1,4) node {0};
\draw[white] (1,5) node {0};
\draw[gray] (2,0) node {2};
\draw[gray] (2,1) node {2};
\draw[gray] (2,2) node {2};
\draw[white] (2,3) node {1};
\draw[white] (2,4) node {1};
\draw[white] (2,5) node {0};
\draw[white] (3,0) node {3};
\draw[white] (3,1) node {0};
\draw[white] (3,2) node {0};
\draw[gray] (3,3) node {2};
\draw[white] (3,4) node {1};
\draw[white] (3,5) node {0};
\draw[white] (4,0) node {1};
\draw[white] (4,1) node {0};
\draw[white] (4,2) node {0};
\draw[gray] (4,3) node {2};
\draw[white] (4,4) node {1};
\draw[white] (4,5) node {0};
\draw[white] (5,0) node {1};
\draw[white] (5,1) node {1};
\draw[white] (5,2) node {3};
\draw[gray] (5,3) node {2};
\draw[white] (5,4) node {1};
\draw[white] (5,5) node {0};
\end{tikzpicture}
\end{center}
 \vspace{3mm}
 \caption{The initial nimbers for the rulesets $\large \symbishop \odot \large \symrook $ (left),  
 $\large\symknight \odot  \large\symrook$ (middle) and 
 $\large \symbishop \odot \large\symknight$ (right).
}\label{fig:BenfR_KenfR_BenfK}
\end{figure}

\noindent {\bf A solution to the problem in Figure~\ref{fig:disjunctive(knight.nim+bishop.nim)}.} Observe from Figure~\ref{fig:BenfR_KenfR_BenfK} that the nimber of $ \large \symknight \odot \large \symrook$ at position $(3,5)$ is $0$  and the nimber of $ \large \symbishop \odot \large \symrook$ at position $(4,3)$ is $3$. Therefore, Alice can win by choosing the combined piece $\large \symbishop \odot \large \symrook$. Whatever rule Bob enforces, Alice can play such that the disjunctive sum nimber is $0$. For instance, if Bob  enforces $ \large \symbishop$, then she plays to $(1,0)$, and otherwise, she plays to $(0,3)$. Therefore, Alice wins, by Theorem~\ref{thm:nimvalue}. Thus, in the first paragraph of the paper, Alice made a  correct move, and thus all Bob's moves are losing. Hence, both played optimally. By continuing from the position in Figure~\ref{fig:disjunctive(knight.nim+bishop.nim)2}, Alice chooses $ \large \symknight \odot \large \symrook$, and plays to either $(1,1)$ or $(2,2)$ depending on which ruleset Bob enforces.\\ 

Let us give one more example from game practice.\\ 
\begin{figure}[htbp!]
\begin{center}
\begin{tikzpicture}[scale = 1,
         box/.style={rectangle,draw=gray!60!yellow,  thick, minimum size=1 cm},
     ]
\foreach \x in {0,1,...,5}{
     \foreach \y in {0,1,...,5}
         \node[box] at (\x,\y){};
 }
\draw[ultra thick, color=gray!60!yellow] (-0.5,-0.5) -- (-0.5, 5.5);
\draw[ultra thick, color=gray!60!yellow] (-0.5, 5.5)-- (5.5, 5.5);
\draw (4,0) node {\large\symbishop \!$\cdot$\!\! \symknight};
\draw[ line width=0.65 pt] (4,0) circle (12.2 pt);
\draw (5,1) node {\Large\symrook};
\draw (0, 5.9) node {$0$};
\draw (-0.9, 5) node {$0$};
\end{tikzpicture}
\end{center}
 \vspace{3mm}
\caption{A disjunctive sum $\large \symbishop \!\odot\!  \large \symknight +\large \symrook$.}\label{fig:BenfKplusR}
\end{figure}
\begin{example}
Who wins if Alice starts the game in Figure~\ref{fig:BenfKplusR}?
First observe from Figure~\ref{fig:WBWK2HN} that the nimber of $\symrook$  at position $(4,5)$ is $1$, and from Figure~\ref{fig:BenfR_KenfR_BenfK} that the nimber of $\large \symbishop \odot \large \symknight$ at position $(5,4)$ is $1$. Therefore, the nimber of the disjunctive sum is $1 \oplus 1= 0$, a $\mathcal{P}$-position. However, if Alice chooses the combined piece, then she can win if Bob enforces the wrong component. Namely, if Bob, by mistake,  enforces the Bishop $\symbishop$, then Alice can move this game piece to the nimber one, at position $(2,1)$, and win. If Alice chooses the Rook $\symrook$, then Bob can reverse to nimber one with the same piece, unless Alice plays to position $(4,4)$ of nimber zero. In this case, Bob can play the combined piece to a nearby nimber zero irrespective of which game piece Alice enforces.
\end{example}

\section{Open problems}\label{sec:openp}
We propose the following open problems. 
We do not yet know if strong domination holds for any non-trivial rulesets.\\
\begin{problem}
    Is there a domination criterion between strong domination and domination that inherits transitivity from strong domination, and is satisfied by  interesting rulesets?\\
\end{problem}

 \begin{problem}
     Are there other `simpler' properties  that reveal when a ruleset dominates another ruleset?\\ 
 \end{problem}

\begin{problem}
    Study the nimbers of individual ruleset combinations with respect to the enforce operator, such as the ones in Section \ref{sec:examples}.\\
\end{problem}

From Figure~\ref{fig:YNWN}, we observe that the nimbers of $\large\rook_{\bf Y}$ coincide with those of $\large\rook_{\bf Y}\odot\large\symqueen$; see also the rulesets $\large \symbishop$ and $\large \symbishop \odot \large \symrook$.\\ 

\begin{problem}
    If a ruleset dominates another, does it also `nimber dominate' that ruleset? That is, does it follow that all nimbers of the combined ruleset are the same as those of the dominating ruleset?\\
\end{problem}

In a partizan setting a ruleset has distinct sets of options for the two players.\\

\begin{problem}
    Study ruleset domination in a partizan setting: when can both players always enforce the same partizan ruleset in optimal play?\\
\end{problem}

\begin{problem}
Study the distributive lattice structure for selective and enforce operator for partizan rulesets. More specifically, when does Theorem~\ref{thm:Absorption-Distributive} hold in a  partizan setting?\\ 
\end{problem}

\vspace{3mm}

\bmhead{Acknowledgments}
This work is supported by JSPS, Grant-in-Aid for Early-Career Scientists, Grant Number 22K13953. 
This work is also supported by JSPS Kakenhi 20K03514, 23K03071, 20H01798.
This work is also supported by JSPS KAKENHI JP21K21283.
This work is also supported by JST, the establishment of university fellowships towards the creation of science technology innovation, Grant Number JPMJFS2129 and JST SPRING, Grant Number JPMJSP2132. Indrajit Saha is partially supported by JSPS KAKENHI Grant Number JP21H04979 and JST ERATO Grant Number JPMJER2301.

\end{document}